\documentclass[12pt,reqno]{amsart}
\usepackage{txfonts}
\usepackage{a4wide}
\allowdisplaybreaks \numberwithin{equation}{section}
\usepackage{color}
\usepackage{exscale}
\usepackage{relsize}
\usepackage{graphicx}
\usepackage[titletoc, title]{appendix}

 \usepackage{verbatim}
\numberwithin{equation}{section}

\newtheorem{theorem}{Theorem}[section]
\newtheorem{proposition}[theorem]{Proposition}

\newtheorem{lemma}[theorem]{Lemma}

\theoremstyle{definition}

\theoremstyle{remark}

\begin{document}
\title[localized nodal solutions for critical $p-$Laplacian equations ]
{Localized nodal solutions for $p-$Laplacian equations with critical exponents in $\mathbb{R}^N$}

 \author{Fengshuang Gao and Yuxia Guo}

\address{  Department of Mathematical Science, Tsinghua University, Beijing, P.R.China}
\email{gfs16@mails.tsinghua.edu.cn}

\address{  Department of Mathematical Science, Tsinghua University, Beijing, P.R.China}
\email{yguo@mail.tsinghua.edu.cn}



\thanks{This work is partially supported by NSFC(11771235)}

\begin{abstract}
In this paper, we consider the existence of localized sign-changing solutions for the $p-$Laplacian nonlinear Schr\"odinger equation
$$
-\epsilon^p\Delta_pu+V(x)|u|^{p-2}u=|u|^{p^*-2}u+\mu|u|^{q-2}u,~~u\in W^{1,p}(\mathbb{R}^N),
$$
where $1<p<N$, $p_N=\max\{p,p^*-1\}<q<p^*=\frac{Np}{N-p}$, $\mu>0$, $\Delta_p$ is the $p-$Laplacian operator. By using the penalization method together with the truncation method and a blow-up argument, we establish for small $\epsilon$ the existence of a sequence of localized nodal solutions concentrating near a given local minimum point of the potential function.
\end{abstract}

\maketitle
\section{introduction}\label{1}

We consider the following $p-$Laplacian equation with critical exponents in $\mathbb{R}^N$:
\begin{equation}\label{1-1}
-\epsilon^p\Delta_pv+V(x)|v|^{p-2}v=|v|^{p^*-2}v+\mu|v|^{q-2}v,~~~~\hbox{in~~} \mathbb{R}^N,
\end{equation}
where $1<p<N$, $p_N=\max\{p,p^*-1\}<q<p^*=\frac{Np}{N-p}$, $\mu>0$, $\Delta_p$ is the $p-$Laplacian operator, $\Delta_pv=\nabla(|\nabla v|^{p-2}\nabla v)$, $V(x)$ satisfies some proper conditions.

 Over the last decades, there has been extensive studies on the existence, multiplicity and qualitative properties for localized solutions of equation \eqref{1-1} with $p=2$, i.e.,
\begin{equation}\label{ad1-1}
-\epsilon^2\Delta v+V(x) v=|v|^{2^*-2}v+|v|^{q-2}v,~~~~\hbox{ in }\mathbb{R}^N.
\end{equation}
 Under various assumptions on the potential $V$, various types of concentration behaviors have been investigated mainly for positive solutions, see \cite{ams,r}. In particular, using Lyapunov-Schmit reduction, the existence of single bump and multi-bump solutions of \eqref{ad1-1} concentrating at each given nondegenerate critical point of $V$ was proved in \cite{fw} and \cite{o1,o2} respectively. Moreover, using a penalization approach, it was shown in \cite{df1} that such solutions can be constructed to concentrate at a local minimum point of $V$ which may be degenerate. Also in \cite{df2} and \cite{g}, the authors glued the single bump positive solutions and obtained multi-bump positive solution at separate local minimum points of $V$. However, in \cite{kw}, Kang and Wei showed that there could not be multi-spiked positive solutions concentrating near a local minimum point. Therefore, in order to have localized solutions near a local minimum, one has to look  for sign-changing solutions.

In this paper, we are intend to construct solutions as localized nodal solutions clustered near the minimum set of $V$. However, even in the semilinear case, i.e., when $p=2$, the powerful Lyapunov-Schmit reduction method is not quite effective due to the fact that little is known about the non-degeneracy of sign-changing solutions  for the corresponding limiting equation, see \cite{mw}. In fact, it turns out this method is very sensitive to the number of positive and negative peaks (\cite{dp2}), which in turn largely limits how many localized sign changing solutions can be constructed. Due to these difficulties, there are a few works concerning this problem, see for example, \cite{as,bcw,dp1,f}. Recently, in \cite{cw}, utilizing a penalization method, Chen and Wang constructed an infinite sequence of localized solutions concentrating near a local minimum point of $V$ for a semilinear subcritical equation. It seems that their method is quite effective and can be generalized to a class of quasilinear Schr\"odinger equation, see \cite{llw1}. Very recently, the result in \cite{cw} was extended in \cite{clw} for the critical exponent case. Following the idea of \cite{clw}, we establish multiple localized nodal solutions for general critical $p-$Laplacian equation.

For this purpose, we introduce the following assumptions on $V(x)$ :
\begin{enumerate}
\item[$(V_1)$] $V\in C^1(\mathbb{R}^N,\mathbb{R})$ and there exists $0<\tilde{a}<\tilde{b}$ such that $\tilde{a}\leq V(x)\leq \tilde{b}$, $\forall x\in\mathbb{R}^N$.
\item[$(V_2)$] There is a bounded domain $\mathcal{M}\subset\mathbb{R}^N$ with smooth boundary $\partial \mathcal{M}$ such that
    $$
    \overrightarrow{n}(x)\cdot\nabla V(x)>0,~~~~\forall x\in\partial\mathcal{M},
    $$
    where $\overrightarrow{n}(x)$ denotes the outward normal to $\partial\mathcal{M}$ at $x$.
\end{enumerate}
Note that $(V_2)$ is satisfied if $V$ has an isolated local minimum set, i.e., $V$ has a local trapping well. Under the assumption $(V_2)$, the critical set of $V$ inside $\mathcal{M}$
$$
\mathcal{A}=\{x\in\mathcal{M}|\nabla V(x)=0\}
$$
is a nonempty compact subset of $\mathcal{M}$. Without loss of generality, we will assume $0\in\mathcal{A}$. For a set $B\subset\mathbb{R}^N$ and $\delta>0$, we denote
$$
B^\delta=\left\{x\in\mathbb{R}^N|dist(x,B)=\inf\limits_{y\in B}|x-y|<\delta\right\},
$$
$$
B_\delta=\left\{x\in\mathbb{R}^N|\delta x\in B\right\}.
$$
With these notations, the main result of this paper is the following:
\begin{theorem}\label{th1-1}
Assume $\max\{p^*-1,p\}<q<p^*$ and $(V_1)$, $(V_2)$ hold. Then for any integer $k$, there exists $\epsilon_k>0$ such that if $0<\epsilon<\epsilon_k$, equation \eqref{1-1} has $k$ pairs of sign-changing solutions $\pm v_{j,\epsilon}$, $j=1, 2, \cdots, k$. Furthermore, there exists a constant $c=c(k)>0$, for any $\delta>0$ there exists $\epsilon_k(\delta)>0$ such that for $0<\epsilon<\epsilon_k(\delta)$,
$$
v_{j,\epsilon}\leq ce^{-\frac{c}{\epsilon}dist(x,\mathcal{A}^\delta)},~~1\leq j\leq k.
$$
\end{theorem}
Now let us outline the method. Define $u(x)=v(\epsilon x)$ then equation \eqref{1-1} is equivalent to
\begin{equation}\label{1-2}
-\Delta_pu+V(\epsilon x)|u|^{p-2}u=|u|^{p^*-2}u+\mu|u|^{q-2}u,~~~~\hbox{in~~} \mathbb{R}^N.
\end{equation}
Solutions of this equation are critical points of the functional
\begin{equation}\label{1-3}
J_\epsilon(u)=\frac{1}{p}\int_{\mathbb{R}^N}|\nabla u|^p+V(\epsilon x)|u|^p-\frac{1}{p^*}\int_{\mathbb{R}^N}|u|^{p^*}-\frac{\mu}{q}\int_{\mathbb{R}^N}|u|^q.
\end{equation}
However, due to the shifts and the dilations, the imbedding $W^{1,p}(\mathbb{R}^N)\hookrightarrow L^{p^*}(\mathbb{R}^N)$ is continuous, but not compact even if locally, hence the standard critical point theory is not able to be applied directly. One classical way to get over this difficulty is to establish a $(PS)_c$ sequence with $c$ smaller than the energy of its corresponding limiting functional. Nevertheless, we are trying to obtain multiple sign-changing solutions which will be obtained as higher topological type critical points of the energy functional at energy levels where compactness condition fails in general. To deal the compactness and localization issue, we adapt a truncation approach in \cite{lz,zll} and a penalization method in \cite{bj,clw} to our case.

To be more precise, let $\varphi\in C_0^\infty(\mathbb{R})$ be such that $\varphi(s)=1$ for $|s|\leq 1$; $\varphi(s)=0$ for $|s|\geq 2$; $|\varphi'(s)|\leq 2$, $\varphi$ is even and decreasing in the interval $[1,2]$. For $\sigma\in (0,1]$, $x\in\mathbb{R}^N$, $t\in\mathbb{R}$, define
$$
\begin{array}{c}
{b_{\sigma}(x, t)=\varphi\left(\sigma e^{dist(\sigma x, M)} t\right),
m_{\sigma}(x, t)=\int_{0}^{t} b_{\sigma}(x, \tau) d \tau}, \\
 {F_{\sigma}(x,t)=\frac{1}{p^{*}}|t|^{r}\left|m_{\sigma}(x,t)\right|^{p^{*}-r}}+\frac{\mu}{q}|t|^r|m_\epsilon(x,t)|^{q-r},
 {f_{\sigma}(x, t)=\frac{\partial}{\partial t} F_{\sigma}(x, t)},
 \end{array}
$$
where $r$ is a fixed number such that
$$
p_N=\max\{p,p^*-1\}<r<q.
$$
For $\sigma=0$, we understand $b_0(x,t)=1$, $m_0(x,t)=t$, $F_0(x,t)=F(t)=\frac{1}{p^*}|t|^{p^*}+\frac{\mu}{q}|t|^q$ and $f_0(x,t)=f(t)=|t|^{p^*-2}t+\mu|t|^{q-2}t$.

 Now we define
\begin{equation}\label{1-7}
I_\epsilon(u)=\frac{1}{p}\int_{\mathbb{R}^N}|\nabla u|^p+V(\epsilon x)|u|^p+Q_\epsilon(u)-\int_{\mathbb{R}^N}F_\epsilon(x,u),
\end{equation}
where $Q_\epsilon(u)$ is a modification of the penalization term in \cite{clw}. In details, let $\zeta\in C_0^\infty(\mathbb{R})$ be a cut-off function such that $0\leq \zeta(t)\leq 1$ and $\zeta'(t)\geq 0$ for every $t\in\mathbb{R}$; $\zeta(t)=0$ for $t\leq 0$, $\zeta(t)>0$ if $t>0$ and $\zeta(t)=1$ if $t\geq 1$. Moreover, there exist constants $C>0$, $a>0$  and $\rho>0$ such that
$$
t\zeta'(t)\geq C\zeta(t),~~~~0\leq t\leq \rho,
$$
and
$$
\zeta(t)\geq Ct^a,~~~~0\leq t\leq \rho.
$$
Let
\begin{equation}\label{1-4}
\xi_\epsilon(x)=\left\{\begin{array}{ll}
0,&\hbox{if}~~x\in M_\epsilon,\\
\epsilon^{-\gamma}\zeta(dist(x,M_\epsilon)),&\hbox{if}~~x\notin M_\epsilon,
\end{array}\right.
\end{equation}
where $\gamma>0$ satisfies
\begin{equation}\label{1-5}
\gamma>\max\left\{\frac{(p+a)p}{N-p},\frac{p^2r}{(r-p)(N-p)}\right\}.
\end{equation}
It is easy to see that  for $\epsilon$ small $\xi_\epsilon$ is a $C^1$ function and
$$
\xi_\epsilon(x)=0~~\hbox{if}~~x\in M_\epsilon,~~~~\xi_\epsilon(x)=\epsilon^{-\gamma}~~\hbox{if}~~x\notin (M_\epsilon)^1.
$$
For $u\in W^{1,p}(\mathbb{R}^N)$, let
\begin{equation}\label{1-6}
Q_\epsilon(u)=\frac{1}{p\beta}\left(\int_{\mathbb{R}^N}\xi_\epsilon|u|^pdx-1\right)_+^\beta,
\end{equation}
where $(t)_+=\max\{t,0\}$, $\beta$ satisfies
$1<\beta<r/p$, and $\beta\leq\frac{1}{2-p}$ if $1<p<2$.

We say $u$ is a critical point of $I_\epsilon$ if for any $\varphi\in W^{1,p}(\mathbb{R}^N)$
\begin{align}\label{1-8}
\langle I'_\epsilon(u),\varphi\rangle=&\int_{\mathbb{R}^N}|\nabla u|^{p-2}\nabla u\nabla \varphi+V(\epsilon x)|u|^{p-2}u\varphi+\left(\int_{\mathbb{R}^N}\xi_\epsilon|u|^p-1\right)_+^{\beta-1}\int_{\mathbb{R}^N}\xi_\epsilon |u|^{p-2}u\varphi\notag\\
&-\int_{\mathbb{R}^N}f_\epsilon(x,u)\varphi.
\end{align}
It is easy to see that the critical point of $I_\epsilon$ is a solution of
\begin{equation}\label{1-9}
\begin{aligned}
-\Delta_p u+V(\epsilon x)|u|^{p-2}u+\left(\int_{\mathbb{R}^N}\xi_\epsilon|u|^p-1\right)_+^{\beta-1}\xi_\epsilon|u|^{p-2}u=f_\epsilon(x,u).
\end{aligned}
\end{equation}
Using a new minimax theorem for sign-changing solutions in \cite{llw-jde,zll1}, see also \cite{blw}, we obtain for any positive integer $k$, there exists $\epsilon_k$ such that the modified functional $I_{\epsilon}$ has at least $k$ pairs of sign-changing critical points $u_{k,\epsilon}$ if $0<\epsilon<\epsilon_k$. From here we need to show these critical points are in fact critical points of the original functional $J_{\epsilon}$, which means we need to have a proper decay estimate of $\{u_{k,\epsilon}\}$ and these solutions can only concentrate in subsets contained in $\mathcal{M}_{\epsilon}$ and having a positive distance away from $\partial \mathcal{M}_{\epsilon}$. To do this, we employ a decomposition result and show that these solutions can only blow up in a domain sufficient close to $\mathcal{M}_{\epsilon}$, then an estimate result in safe domain and a local Pohozaev identity rule out the possibility of bubble blow-ups, see also \cite{cpy,ds1,gv,zll}, which in turns gives the uniform bound result. Finally, another Pohozaev identity yields that these solutions concentrate in $\mathcal{A}_\epsilon$ and satisfy a proper decay estimate, so these are solutions of the original problem.

The paper is organized as follows. In section \ref{2}, we prove the $(PS)$ condition for $I_{\epsilon}$, with which we can deduce the existence of infinitely many solutions for \eqref{1-9} in section \ref{3}. Then in section \ref{4}, we give a uniform bound for the solutions obtained in section \ref{3}, and finally prove Theorem \ref{th1-1} in section \ref{5}.

Throughout this paper, we will denote the norms of $W^{1,p}(\mathbb{R}^N)$ and $L^p(\mathbb{R}^N)$ by $\|\cdot\|$ and $\|\cdot\|_p$ respectively.
\section{The Palais-Smale condition for modified functional}\label{2}
\begin{proposition}\label{prop2-1}
The function $b_\epsilon$, $m_\epsilon$ and $f_\epsilon$ satisfy the following properties
\begin{enumerate}
\item $0\leq b_\epsilon(x,s)\leq 1$.
\item $m_\epsilon(x,s)=s$ if $|s|<\epsilon^{-1}e^{-dist(\epsilon x, M)}$.
\item $sm_\epsilon(x,s)\geq 0$, $|m_\epsilon(x,s)|\geq |s|b_\epsilon(x,s)$.
\item $\min\left\{|s|,\epsilon^{-1}e^{-dist(\epsilon x,M)}\right\}\leq |m_\epsilon(x,s)|\leq\min\left\{|s|,2\epsilon^{-1}e^{-dist(\epsilon x,M)}\right\}$.
\item $|f_\epsilon(x,s)|\leq |s|^{r-1}|m_\epsilon(x,s)|^{p^*-r}+\mu|s|^{r-1}|m_\epsilon(x,s)|^{q-r}\leq |s|^{p^*-1}+\mu|s|^{q-1}$.
\item $\frac{1}{r}sf_\epsilon(x,s)-F_\epsilon(x,s)=\frac{p^*-r}{p^*r}|s|^{r+1}|m_\epsilon(x,s)|^{p^*-r-1}b_\epsilon(x,s)+\frac{\mu(q-r)}{qr}|s|^{r+1}|m_\epsilon(x,s)|^{q-r-1}b_\epsilon(x,s).$
\item
$F_\epsilon(x,s)-\frac{1}{p^*}sf_\epsilon(x,s)=\frac{p^*-r}{(p^*)^2}|s|^r|m_\epsilon(x,s)|^{p^*-r-1}(|m_\epsilon(x,s)-|s|b_\epsilon(x,s))$\\
$+\frac{\mu(q-r)}{p^*q}|s|^r|m_\epsilon(x,s)|^{q-r-1}\left(|m_\epsilon(x,s)|-|s|b_\epsilon(x,s)\right)+\left(\frac{1}{q}-\frac{1}{p^*}\right)\mu|s|^r|m_\epsilon(x,s)|^{q-r}$.
\item $\nabla_x F_\epsilon(x,s)=-\epsilon\nabla dist(\epsilon x,M)\left(\frac{p^*-r}{p^*}|s|^r|m_\epsilon(x,s)|^{p^*-r-1}+\frac{\mu(q-r)}{q}|s|^r|m_\epsilon(x,s)|^{q-r-1}\right)$\\

    \hspace{50pt}$\times\left(|m_\epsilon(x,s)|-|s|b_\epsilon(x,s)\right)$.
\end{enumerate}
\end{proposition}
\begin{proof}
The proof is similar to Lemma 2.1 in \cite{zll}, we omit it here.
\end{proof}
\begin{lemma}\label{le2-1}
 For every fixed $\epsilon$, $I_\epsilon$ satisfies $(PS)_c$ condition.
\end{lemma}

\begin{proof}
Let $\{u_n\}$ be a $(PS)_c$ sequence, take $\varphi=u_n$ in equation \eqref{1-8}, by (6) in Proposition \ref{prop2-1}, we have
\begin{equation}\label{2-1}
\begin{aligned}
c+o(\|u_n\|)\geq& I_\epsilon(u_n)-\frac{1}{r}\langle DI_\epsilon(u_n),u_n\rangle\\
=&\left(\frac{1}{p}-\frac{1}{r}\right)\int_{\mathbb{R}^N}|\nabla u_n|^p+V(\epsilon x)|u_n|^p+\int_{\mathbb{R}^N}\frac{1}{r}f_\epsilon(x,u_n)u_n-F_\epsilon(x,u_n)\\
&+\frac{1}{p\beta}\left(\int_{\mathbb{R}^N}\xi_\epsilon(x)|u_n|^p-1\right)_+^{\beta}-\frac{1}{r}\left(\int_{\mathbb{R}^N}\xi_\epsilon(x)|u_n|^p-1\right)_+^{\beta-1}\int_{\mathbb{R}^N}\xi_\epsilon|u_n|^p\\
\geq&\left(\frac{1}{p}-\frac{1}{r}\right)\int_{\mathbb{R}^N}|\nabla u_n|^p+V(\epsilon x)|u_n|^p+\frac{1}{p\beta}\left(\int_{\mathbb{R}^N}\xi_\epsilon(x)|u_n|^p-1\right)_+^{\beta}\\
&-\frac{1}{r}\left(\int_{\mathbb{R}^N}\xi_\epsilon(x)|u_n|^p-1\right)_+^{\beta-1}\int_{\mathbb{R}^N}\xi_\epsilon|u_n|^p.
\end{aligned}
\end{equation}
Since $p<p\beta<r$, we get that there exists $\eta_M>0$ independent of $\epsilon$ such that $\|u_n\|\leq \eta_M$ and $Q_\epsilon(u_n)\leq\eta_M$. We assume that, up to a subsequence, $u_n\rightharpoonup u$ in $W^{1,p}(\mathbb{R}^N)$ as $n\to\infty$ and
\begin{equation}\label{2-2}
\lambda_n=\left(\int_{\mathbb{R}^N}\xi_\epsilon|u_n|^p-1\right)_+^{\beta-1}\to\lambda,~~n\to\infty.
\end{equation}
Then we have as $n,m\to\infty$
\begin{equation}\label{2-3}
\begin{aligned}
o(1)=&\langle DI_\epsilon (u_n)-DI_\epsilon(u_m),u_n-u_m\rangle\\
=&\int_{\mathbb{R}^N}\left(|\nabla u_n|^{p-2}\nabla u_n-|\nabla u_m|^{p-2}\nabla u_m,\nabla u_n-\nabla u_m\right)+V(\epsilon x)\left(|u_n|^{p-2}u_n-|u_m|^{p-2}u_m,u_n-u_m\right)\\
&+\lambda_n\int_{\mathbb{R}^N}\xi_\epsilon|u_n|^{p-2}u_n(u_n-u_m)-\lambda_m\int_{\mathbb{R}^N}\xi_\epsilon|u_m|^{p-2}u_m(u_n-u_m)\\
&-\int_{\mathbb{R}^N}(f_\epsilon(x,u_n)-f_\epsilon(x,u_m))(u_n-u_m).\\
\end{aligned}
\end{equation}
In order to estimate \eqref{2-3}, the following elementary inequalities are very useful, see \cite{lz}. For $p>1$ there exists a constant $c_p$ such that for $\xi,\eta\in\mathbb{R}^N$, we have
\begin{equation}\label{2-4}
\begin{array}{ll}
{\left(|\xi|^{p-2} \xi-|\eta|^{p-2} \eta, \xi-\eta\right) \geq c_{p}|\xi-\eta|^{p}}, & {\text { if } p \geq 2}, \\
 {\left(|\xi|^{p-2} \xi-|\eta|^{p-2} \eta, \xi-\eta\right) \geq c_{p}|\xi-\eta|^{2}\left(|\xi|^{2-p}+|\eta|^{2-p}\right)^{-1}}, & {\text { if } 1<p<2}.
 \end{array}
\end{equation}
For $p\geq 2$, we obtain
\begin{equation}\label{2-5}
 \int_{\mathbb{R}^{N}}\left(\left|\nabla u_{n}\right|^{p-2} \nabla u_{n}-\left|\nabla u_{m}\right|^{p-2} \nabla u_{m}, \nabla u_{n}-\nabla u_{m}\right)d x\geq
C\int_{\mathbb{R}^N}\left|\nabla u_{n}-\nabla u_{m}\right|^{p} d x.
\end{equation}
For $1<p<2$, we obtain
\begin{equation}\label{2-6}
\begin{aligned}
&\int_{\mathbb{R}^{N}}\left(\left|\nabla u_{n}\right|^{p-2} \nabla u_{n}-\left|\nabla u_{m}\right|^{p-2} \nabla u_{m}, \nabla u_{n}-\nabla u_{m}\right) d x\\
\geq& C\left(\int_{\mathbb{R}^N}\left(\left|\nabla u_{n}\right|^{p-2} \nabla u_{n}-\left|\nabla u_{m}\right|^{p-2} \nabla u_{m}, \nabla u_{n}-\nabla u_{m}\right)
dx\right)\left(\int_{\mathbb{R}^N}\left(\left|\nabla u_{n}\right|^{p}+\left|\nabla u_{m}\right|^{p}\right) d x\right)^{\frac{2-p}{p}} \\
\geq& C\left(\int_{\mathbb{R}^N}\left|\left(\left|\nabla u_{n}\right|^{p-2} \nabla u_{n}-\left|\nabla u_{m}\right|^{p-2} \nabla u_{m}, \nabla u_{n}-\nabla u_{m}\right)\right|^{\frac{p}{2}}\left(\left|\nabla u_{n}\right|^{2-p}+\left|\nabla u_{m}\right|^{2-p}\right)^{\frac{p}{2}} dx\right)^{\frac{2}{p}}\\
 \geq& C\left(\int_{\mathbb{R}^N}\left|\nabla u_{n}-\nabla u_{m}\right|^{p} d x\right)^{\frac{2}{p}}.
\end{aligned}
\end{equation}
Similarly, we have
\begin{equation}\label{ad2-5}
\begin{aligned}
\int_{\mathbb{R}^N}V(\epsilon x)(|u_n|^{p-2}u_n-|u_m|^{p-2}u_m,u_n-u_m)\geq\left\{\begin{array}{ll}
C\|u_n-u_m\|_p^p,&\hbox{if}~~p\geq2,\\
C\|u_n-u_m\|_p^{2},&\hbox{if}~~1<p<2.
\end{array}\right.
\end{aligned}
\end{equation}
From \eqref{2-2}, we get that
\begin{equation}\label{2-7}
\begin{aligned}
&\lambda_n\int_{\mathbb{R}^N}\xi_\epsilon|u_n|^{p-2}u_n(u_n-u_m)-\lambda_m\int_{\mathbb{R}^N}\xi_\epsilon|u_m|^{p-2}u_m(u_n-u_m)\\
=&\lambda\int_{\mathbb{R}^N}\xi_\epsilon\left(|u_n|^{p-2}u_n-|u_m|^{p-2}u_m\right)\left(u_n-u_m\right)+o(1),~~~~\hbox{as}~~n,m\to\infty.
\end{aligned}
\end{equation}
Using Proposition \ref{prop2-1}, we deduce that
\begin{equation}\label{2-8}
\begin{aligned}
&\left|\int_{\mathbb{R}^N}\left(f_\epsilon(x,u_n)-f_\epsilon(x,u_m)\right)\left(u_n-u_m\right)\right|\\
\leq&C\int_{\mathbb{R}^N}\left(\left(2\epsilon^{-1}e^{-dist(\epsilon x,M)}\right)^{q-r}+\left(2\epsilon^{-1}e^{-dist(\epsilon x,M)}\right)^{p^*-r}\right)\left(|u_n|^{r-1}+|u_m|^{r-1}\right)|u_n-u_m|\\
\leq&Ce^{-(q-r)R}\int_{dist(\epsilon x,M)\geq R}\left(|u_n|^r+|u_m|^r\right)+C\left(\int_{dist(\epsilon x,M)<R}|u_n|^r+|u_m|^r\right)^{\frac{r-1}{r}}\left(\int_{dist(\epsilon x,M)<R}|u_n-u_m|^r\right)^{\frac{1}{r}}\\
\leq& Ce^{-(q-r)R}+C\left(\int_{dist(\epsilon x,M)<R}|u_n-u_m|^r\right)^{\frac{1}{r}}\to 0,~~~~\hbox{as}~~m,n\to\infty.
\end{aligned}
\end{equation}
Combining \eqref{2-3}, \eqref{2-5}, \eqref{2-6}, \eqref{ad2-5}, \eqref{2-7}, \eqref{2-8}, we obtain that
\begin{equation}\label{2-10}
\begin{aligned}
o(1)=&\langle DI_\epsilon (u_n)-DI_\epsilon(u_m),u_n-u_m\rangle\\
\geq&\left\{\begin{array}{ll}
C\|u_n-u_m\|^p-o(1),&\hbox{if}~~p\geq2,\\
C\|u_n-u_m\|^{2}-o(1),&\hbox{if}~~1<p<2,
\end{array}\right.
\end{aligned}
\end{equation}
which yields that for fixed $\epsilon>0$, $\|u_n-u_m\|\to 0$ as $n,m\to\infty$.
\end{proof}
\section{Existence of multiple sign-changing critical points of $I_\epsilon$}\label{3}
In this section, we construct multiple nodal critical points of the perturbed functionals. To obtain multiple sign-changing critical points of $I_\epsilon$, we adapt an abstract critical point theorem in \cite{llw-jde}. For reader's convenience, we state it here.

Let $X$ be a Banach space, $f$ be an even $C^1$ functional on $X$. Let $\tilde{P}$, $\tilde{Q}$ be open convex sets  of $X$, $\tilde{Q}=-\tilde{P}$. Set
$$
O=\tilde{P}\cup \tilde{Q},~~~~\Sigma=\partial \tilde{P}\cap\partial \tilde{Q}.
$$
Assume
\begin{enumerate}
\item[$(I_1)$] $f$ satisfies the Palais-Smale condition.
\item[$(I_2)$] $c^*=\inf\limits_{x\in\Sigma}f(x)>0$.
\end{enumerate}
Assume there exists an odd continuous map $A:X\to X$ satisfying
\begin{enumerate}
\item[$(A_1)$] Given $c_0,b_0>0$, there exists $b=b(c_0,b_0)>0$ such that if $\|Df(x)\|\geq b_0$, $|f(x)|\leq c_0$, then
    $$
    \langle Df(x),x-Ax\rangle\geq b\|x-Ax\|>0.
    $$
\item[$(A_2)$] $A(\partial \tilde{P})\subset \tilde{P}$, $A(\partial \tilde{Q})\subset \tilde{Q}$.
\end{enumerate}
Define
$$
\Gamma_j=\{E|E\subset X, E \hbox{ compact}, -E=E, \gamma(E\cap\eta^{-1}(\Sigma))\geq j \hbox{ for } \eta \in\Lambda\},
$$
$$
\Lambda=\{\eta|\eta\in C(X,X), \eta\hbox{ odd}, \eta(P)\subset P, \eta(Q)\subset Q, \eta(x)=x \hbox{ if } f(x)<0\},
$$
where $\gamma$ is the genus of symmetric sets
$$
\gamma(E)=\inf\left\{n|\hbox{there exists an odd map }\varphi:E\to\mathbb{R}^n\setminus\{0\}\right\}.
$$
Assume that
\begin{enumerate}
\item[$(\Gamma)$] $\Gamma_j$ is nonempty.
\end{enumerate}
Define
$$
c_j=\inf\limits_{A\in\Gamma_j}\sup\limits_{x\in A\setminus O}f(x),~~~~~j=1,2,\cdots,
$$
$$
K_c=\{x|Df(x)=0,f(x)=c\},~~K_c^*=K_c\setminus O.
$$
The following abstract critical point theorem is from \cite{llw-jde}:
\begin{theorem}\label{th3-1}
Assume  $(I_1)$, $(I_2)$, $(A_1)$, $(A_2)$, $(\Gamma)$ hold. Then
\begin{enumerate}
\item $c_j\geq c^*$, $K^*_{c_j}\neq \emptyset.$
\item $c_j\to\infty$, as $j\to\infty$.
\item If $c_j=c_{j+1}=\cdots=c_{j+k-1}=c$, then $\gamma(K_c^*)\geq k$.
\end{enumerate}
\end{theorem}
In the following, we verify that the functional $I_\epsilon$ satisfies all the assumptions of Theorem \ref{th3-1}. In Lemma \ref{le2-1}, we have proved that $I_\epsilon$ satisfies the assumption $(I_1)$, i.e., the Palais-Smale condition.

Let
$$
P=\{u|u\in W^{1,p}(\mathbb{R}^N), u\geq 0, a.e. x\in\mathbb{R}^N\},~~~~Q=-P,
$$
$$
P^\delta=\{u|u\in W^{1,p}(\mathbb{R}^N), dist(u,P)<\delta\},
$$
$$
Q^\delta=-P^\delta=\{u|u\in W^{1,p}(\mathbb{R^N}), dist(u,Q)<\delta\}.
$$
It is easy to check that there exists $\delta_0(\epsilon)$ and $c^*(\epsilon, \delta)$, such that for $\delta<\delta_0$, $I_\epsilon(u)\geq c^*$, for $u\in \partial P^\delta\cap \partial Q^\delta$. In fact
\begin{equation}\label{ad3-1}
\begin{aligned}
I_\epsilon(u)=&\frac{1}{p}\int_{\mathbb{R}^N}|\nabla u|^p+V(\epsilon x)|u|^p+\frac{1}{p\beta}\left(\int_{\mathbb{R}^N}\xi_\epsilon(x)|u|^p-1\right)^\beta_+-\int_{\mathbb{R}^N}F_\epsilon(x,u)\\
&\geq C\int_{\mathbb{R}^N}|\nabla u|^p+|u|^p-C\int_{\mathbb{R}^N}|u|^{p^*}-C\int_{\mathbb{R}^N}|u|^q\\
&\geq C\|u\|^p-C\|u\|^{p^*}-C\|u\|^{q},
\end{aligned}
\end{equation}
choose $\delta_0$ small enough, we can see it from \eqref{ad3-1}.

We define the operator $A: W^{1,p}(\mathbb{R}^N)\to W^{1,p}(\mathbb{R}^N)$. Given $u\in W^{1,p}(\mathbb{R}^N)$, define $v=Au\in W^{1,p}(\mathbb{R}^N)$ by the following equation
\begin{equation}\label{3-1}
-\Delta_p v+V(\epsilon x)|v|^{p-2}v+\left(\int_{\mathbb{R}^N}\xi_\epsilon|u|^p-1\right)_+^{\beta-1}\xi_\epsilon|z|^{p-2}v=f_\epsilon(x,u),
\end{equation}
where $z=v$ for $1<p<2$; $z=u$ for $p\geq 2$.
\begin{lemma}\label{le3-1}
$A$ is well defined and continuous on $W^{1,p}(\mathbb{R}^N)$.
\end{lemma}
\begin{proof}
We prove $v=Au$ can be obtained by solving the following minimization problem:
$$
\inf\{G(v)|v\in W^{1,p}(\mathbb{R}^N)\},
$$
where
\begin{equation}\label{3-2}
\begin{aligned}
G(v)=&\frac{1}{p}\int_{\mathbb{R}^N}|\nabla v|^p+V(\epsilon x)|v|^p-\int_{\mathbb{R}^N}f_\epsilon (x,u)v\\
&+\left\{\begin{array}{ll}
\frac{1}{p}\left(\int_{\mathbb{R}^N}\xi_\epsilon|u|^p-1\right)_+^{\beta-1}\int_{\mathbb{R^N}}\xi_\epsilon|v|^p,&1<p<2,\\
\frac{1}{2}\left(\int_{\mathbb{R}^N}\xi_\epsilon|u|^p-1\right)_+^{\beta-1}\int_{\mathbb{R^N}}\xi_\epsilon|u|^{p-2}|v|^2,&p\geq 2.
\end{array}\right.
\end{aligned}
\end{equation}
Notice that
\begin{align*}
\int_{\mathbb{R}^N}f_\epsilon(x,u)v\leq&\left(\int_{\mathbb{R}^N}(|u|^{p^*-1})^{\frac{p^*}{p^*-1}}\right)^{\frac{p^*-1}{p^*}}\left(\int_{\mathbb{R}^N}|v|^{p^*}\right)^{\frac{1}{p^*}}+\mu\left(\int_{\mathbb{R}^N}(|u|^{q-1})^{\frac{q}{q-1}}\right)^{\frac{q-1}{q}}\left(\int_{\mathbb{R}^N}|v|^{q}\right)^{\frac{1}{q}}\\
\leq& C\|v\|,
\end{align*}
we have
\begin{equation}\label{3-3}
\begin{aligned}
G(v)\geq& C\int_{\mathbb{R}^N}|\nabla v|^p+|v|^p-\int_{\mathbb{R}^N}f_\epsilon(x,u)v\\
\geq& C\|v\|^p-C\|v\|,
\end{aligned}
\end{equation}
thus $G$ is coercive and weakly lower semicontinuous. Assume  $v_1$, $v_2$ are two solutions, then
\begin{equation}\label{3-4}
\begin{aligned}
&\langle DG(v_1)-DG(v_2), v_1-v_2\rangle\\
=&\int_{\mathbb{R}^N}\left(|\nabla v_1|^{p-2}\nabla v_1-|\nabla v_2|^{p-2}\nabla v_2\right)\left(\nabla v_1-\nabla v_2\right)+V(\epsilon x)\left(|v_1|^{p-2}v_1-|v_2|^{p-2}v_2\right)\left(v_1-v_2\right)\\
&+\left\{\begin{array}{ll}
\left(\int_{\mathbb{R}^N}\xi_\epsilon|u|^p-1\right)_+^{\beta-1}\int_{\mathbb{R}^N}\xi_\epsilon|u|^{p-2}(v_1-v_2)^2, &p\geq 2,\\
\left(\int_{\mathbb{R}^N}\xi_\epsilon|u|^p-1\right)_+^{\beta-1}\int_{\mathbb{R}^N}\xi_\epsilon(|v_1|^{p-2}v_1-|v_2|^{p-2}v_2)(v_1-v_2), &1<p<2.
\end{array}\right.
\end{aligned}
\end{equation}
As in \eqref{2-10}, we obtain
\begin{equation}\label{3-5}
\begin{aligned}
\langle DG(v_1)-DG(v_2), v_1-v_2\rangle\geq\left\{\begin{array}{ll}
C\|v_1-v_2\|^p,&\hbox{if}~~p\geq2,\\
C\|v_1-v_2\|^{2},&\hbox{if}~~1<p<2,
\end{array}\right.
\end{aligned}
\end{equation}
as a result, \eqref{3-1} has a unique solution $v=Au$.

Now we show that the map $A$ is continuous. First of all, $A$ maps bounded sets into bounded sets. In fact, let $v=Au$, then use $v$ as a test function in \eqref{3-1}, by $(5)$ in Proposition \ref{prop2-1}, we have
\begin{equation}\label{3-6}
\begin{aligned}
C\|v\|^p\leq&\int_{\mathbb{R}^N}|\nabla v|^p+V(\epsilon x)|v|^p+\left(\int_{\mathbb{R}^N}\xi_\epsilon|u|^p-1\right)^{\beta-1}_+\int_{\mathbb{R}^N}\xi_\epsilon|z|^{p-2}v^2\\
=&\int_{\mathbb{R}^N}f_\epsilon(x,u)v\leq C\int_{\mathbb{R}^N}\left(|u|^{p^*-1}+|u|^{q-1}\right)|v|\\
\leq&C\left(\|u\|^{p^*-1}+\|u\|^{q-1}\right)\|v\|.
\end{aligned}
\end{equation}
Let $v_i=Au_i$,
\begin{equation}\label{3-7}
\begin{aligned}
&\int_{\mathbb{R}^N}\left(|\nabla v_1|^{p-2}\nabla v_1-|\nabla v_2|^{p-2}\nabla v_2, \nabla v_1-\nabla v_2\right)+V(\epsilon x)\left(|v_1|^{p-2}v_1-|v_2|^{p-2}v_2,v_1-v_2\right)\\
=&\int_{\mathbb{R}^N}\left(f_\epsilon(x,u_1)-f_\epsilon (x,u_2)\right)(v_1-v_2)-\lambda_{u_1}\int_{\mathbb{R}^N}\xi_\epsilon|z_1|^{p-2}v_1(v_1-v_2)\\
&+\lambda_{u_2}\int_{\mathbb{R}^N}\xi_\epsilon |z_2|^{p-2}v_2(v_1-v_2),
\end{aligned}
\end{equation}
where $z_i=u_i$ if $p\geq 2$; $z_i=v_i$ if $1<p<2$.

For all $\varepsilon>0$, if $\|u_1-u_2\|\leq\varepsilon$, there exists $v\in L^{p^*}$, $w\in L^{q}$ such that $|f_\epsilon(x,u_1)-f_\epsilon (x,u_2)|\leq v+w$, and $\|v\|_{p^*}+\|w\|_{q}\leq 2\varepsilon$, see \cite{w}. Thus
\begin{equation}\label{3-8}
\begin{aligned}
&\int_{\mathbb{R}^N}\left(f_\epsilon(x,u_1)-f_\epsilon (x,u_2)\right)(v_1-v_2)\\
\leq&\|v\|_{p^*}^{p^*-1}\|v_1-v_2\|_{p^*}+\|w\|_q^{q-1}\|v_1-v_2\|_q=o(1)\|v_1-v_2\|.
\end{aligned}
\end{equation}
For $1<p<2$, by the definition of $\lambda_{u_i}$, we have
\begin{equation}\label{3-9}
\begin{aligned}
&-\lambda_{u_1}\int_{\mathbb{R}^N}\xi_\epsilon|v_1|^{p-2}v_1(v_1-v_2)+\lambda_{u_2}\int_{\mathbb{R}^N}\xi_\epsilon |v_2|^{p-2}v_2(v_1-v_2)\\
\leq&-(\lambda_{u_1}-\lambda_{u_2})\int_{\mathbb{R}^N}\xi_\epsilon|v_1|^{p-2}v_1(v_1-v_2)-\lambda_{u_2}\int_{\mathbb{R}^N}\xi_\epsilon\left(|v_1|^{p-2}v_1-|v_2|^{p-2}v_2\right)(v_1-v_2)\\
=&o(1)\|v_1-v_2\|.
\end{aligned}
\end{equation}
For $p\geq 2$, we have
\begin{equation}\label{3-10}
\begin{aligned}
&-\lambda_{u_1}\int_{\mathbb{R}^N}\xi_\epsilon|u_1|^{p-2}v_1(v_1-v_2)+\lambda_{u_2}\int_{\mathbb{R}^N}\xi_\epsilon |u_2|^{p-2}v_2(v_1-v_2)\\
\leq&-\lambda_{u_1}\int_{\mathbb{R}^N}\xi_\epsilon\left(|u_1|^{p-2}-|u_2|^{p-2}\right)v_1(v_1-v_2)-\lambda_{u_1}\int_{\mathbb{R}^N}\xi_\epsilon|u_2|^{p-2}(v_1-v_2)^2\\
&-(\lambda_{u_1}-\lambda_{u_2})\int_{\mathbb{R}^N}\xi_\epsilon|u_2|^{p-2}v_2(v_1-v_2)\\
=&o(1)\|v_1-v_2\|.
\end{aligned}
\end{equation}
Combining \eqref{2-5}, \eqref{2-6}, \eqref{ad2-5} with \eqref{3-7} \eqref{3-8},\eqref{3-9}, \eqref{3-10}, we obtain the desired.
\end{proof}
\begin{lemma}\label{le3-2}
There exists $\delta_0>0$ such that if $0<\delta<\delta_0$, then
$$
A(\partial P^\delta)\subset P^\delta,~~~~A(\partial Q^\delta)\subset Q^\delta.
$$
\end{lemma}
\begin{proof}
For $u\in W^{1,p}(\mathbb{R}^N)$, $u\in \partial (P^\delta)$, denote $v=Au$. Take $\varphi=v^-$ as a test function in \eqref{3-1}, then
\begin{equation}\label{3-11}
\begin{aligned}
&(dist_{W^{1,p}}(v,P))^{p-1}\|v^-\|\leq\|v^-\|^p\leq C\int_{\mathbb{R}^N}|\nabla v^-|^p+V(\epsilon x)|v^-|^p\\
=&C\int_{\mathbb{R}^N}|\nabla v|^{p-2}\nabla v\nabla v^-+V(\epsilon x)|v|^{p-2}vv^-\\
=&-C\left(\int_{\mathbb{R}^N}\xi_\epsilon|u|^p-1\right)_+^{\beta-1}\int_{\mathbb{R}^N}\xi_\epsilon|z|^{p-2}vv^-+C\int_{\mathbb{R}^N}f_\epsilon(x,u)v^-\\
\leq&C\int_{\mathbb{R}^N}f_\epsilon(x,u)v^-\\
\leq&C\int_{\mathbb{R}^N}f_\epsilon(x,u^-)v^-\\
\leq&\left(\int_{\mathbb{R}^N}|u^-|^{p^*}\right)^{\frac{p^*-1}{p^*}}\left(\int_{\mathbb{R}^N}|v^-|^{p^*}\right)^{\frac{1}{p^*}}+\mu\left(\int_{\mathbb{R}^N}|u^-|^q\right)^{\frac{q-1}{q}}\left(\int_{\mathbb{R}^N}|v^-|^q\right)^{\frac{1}{q}}\\
=&\left(\left(dist_{L^{p^*}}(u,P)\right)^{p^*-1}+\mu\left(dist_{L^q}(u,P)\right)^{q-1}\right)\|v^-\|\\
\leq&\left(\left(dist_{W^{1,p}}(u,P)\right)^{p^*-1}+\mu\left(dist_{W^{1,p}}(u,P)\right)^{q-1}\right)\|v^-\|,
\end{aligned}
\end{equation}
thus $A(\partial P^\delta)\subset P^\delta$. In the same way, we can prove $A(\partial Q^\delta)\subset Q^\delta$.
\end{proof}
\begin{lemma}\label{le3-3}
Given $c_0, b_0>0$, there exists $b=b(c_0,b_0)$ such that if $\|DI_\epsilon(u)\|\geq b_0$, $|I_\epsilon(u)|\leq c_0$, then
$$
\langle DI_\epsilon(u), u-Au\rangle\geq b\|u-Au\|.
$$
\end{lemma}
\begin{proof}
Denote $v=Au$, for $\varphi\in W^{1,p}(\mathbb{R}^N)$
\begin{equation}\label{3-12}
\begin{aligned}
\langle DI_\epsilon(u), \varphi\rangle
=&\int_{\mathbb{R}^N}|\nabla u|^{p-2}\nabla u\nabla \varphi+V(\epsilon x)|u|^{p-2}u\varphi+\left(\int_{\mathbb{R}^N}\xi_\epsilon |u|^p-1\right)_+^{\beta-1}\int_{\mathbb{R}^N}\xi_\epsilon(x)|u|^{p-2}u\varphi\\
&-\int_{\mathbb{R}^N}f_\epsilon(x,u)\varphi\\
=&\int_{\mathbb{R}^N}\left(|\nabla u|^{p-2}\nabla u-|\nabla v|^{p-2}\nabla v\right)\nabla \varphi+V(\epsilon x)\left(|u|^{p-2}u-|v|^{p-2}v\right)\varphi\\
&+\left(\int_{\mathbb{R}^N}\xi_\epsilon |u|^p-1\right)_+^{\beta-1}\int_{\mathbb{R}^N}\xi_\epsilon(x)\left(|u|^{p-2}u-|z|^{p-2}v\right)\varphi.
\end{aligned}
\end{equation}
Take $\varphi=u-v$, by \eqref{2-4} and reverse H\"older's inequality, we have that
\begin{equation}\label{3-13}
\begin{aligned}
\langle DI_\epsilon(u), u-v\rangle
\geq\left\{\begin{array}{ll}
C\|u-v\|^2\left(\|u\|^{2-p}+\|v\|^{2-p}\right)^{-1},&\hbox{if } 1<p<2;\\
C\|u-v\|^p,&\hbox{if } p\geq 2.
\end{array}\right.
\end{aligned}
\end{equation}
Moreover, denote
$$
\lambda_\epsilon=\left(\int_{\mathbb{R}^N}\xi_\epsilon |u|^p-1\right)_+^{\beta-1},
$$
we see that
\begin{equation}\label{ad3-13}
\begin{aligned}
\langle DI_\epsilon, u-v\rangle
\geq\left\{\begin{array}{ll}
\lambda_\epsilon\int_{\mathbb{R}^N}\xi_\epsilon(x)\left(|u|^{p-2}u-|v|^{p-2}v\right)(u-v),&\hbox{if } 1<p<2;\\
\lambda_\epsilon\int_{\mathbb{R}^N}\xi_\epsilon(x)|u|^{p-2}(u-v)^2,&\hbox{if } p\geq 2.
\end{array}\right.
\end{aligned}
\end{equation}

On the other hand, applying the following elementary inequalities in \cite{zll1}, that is, for $\xi, \eta\in\mathbb{R}^N$
\begin{equation}\label{3-14}
\begin{aligned}
\left\{\begin{array}{ll}
\left.|| \xi\right|^{p-2} \xi-|\eta|^{p-2} \eta|\leq c| \xi-\left.\eta\right|^{p-1},&\hbox{if } 1<p<2;\\
\|\left.\xi\right|^{p-2} \xi-|\eta|^{p-2} \eta\left|\leq c\left(|\xi|^{p-2}+|\eta|^{p-2}\right)\right| \xi-\eta |, &\hbox{if } p\geq2.
\end{array}\right.
\end{aligned}
\end{equation}
If $1<p<2$, we have
\begin{equation}\label{3-15}
\begin{aligned}
\langle DI_\epsilon, \varphi\rangle&\leq\int_{\mathbb{R}^N}|\nabla u-\nabla v|^{p-1}|\nabla\varphi|+V(\epsilon x)|u-v|^{p-1}|\varphi|\\
&+\left(\int_{\mathbb{R}^N}\xi_\epsilon |u|^p-1\right)_+^{\beta-1}\int_{\mathbb{R}^N}\xi_\epsilon(x)\left(|u|^{p-2}u-|v|^{p-2}v\right)\varphi,
\end{aligned}
\end{equation}
considering \eqref{ad3-13} and \eqref{3-14}, we have
\begin{equation}\label{3-16}
\begin{aligned}
&\lambda_\epsilon\int_{\mathbb{R}^N}\xi_\epsilon(x)\left(|u|^{p-2}u-|v|^{p-2}v\right)\varphi\\
\leq & \lambda_\epsilon\left(\int_{\mathbb{R}^N}\xi_\epsilon\frac{\left||u|^{p-2}u-|v|^{p-2}v\right|}{|u-v|^{p-1}}|\varphi|^p\right)^{\frac{1}{p}}\left(\int_{\mathbb{R}^N}\xi_\epsilon\left(|u|^{p-2}u-|v|^{p-2}v\right)(u-v)\right)^{1-\frac{1}{p}}\\
\leq&C\lambda_\epsilon\left(\int_{\mathbb{R}^N}|\varphi|^p\right)^{\frac{1}{p}}\left(\langle DI_\epsilon(u),u-v\rangle\lambda_\epsilon^{-1}\right)^{1-\frac{1}{p}}\\
\leq&C\|\varphi\|\langle DI_\epsilon(u),u-v\rangle^{1-\frac{1}{p}}\lambda_\epsilon^{\frac{1}{p}}.
\end{aligned}
\end{equation}
While for $p\geq 2$, we have that
\begin{equation}\label{3-17}
\begin{aligned}
\langle DI_\epsilon, \varphi\rangle&\leq\int_{\mathbb{R}^N}(|\nabla u|^{p-2}+|\nabla v|^{p-2})|\nabla u-\nabla v||\nabla\varphi|+V(\epsilon x)(|u|^{p-2}+|v|^{p-2})|u-v||\varphi|\\
&+\left(\int_{\mathbb{R}^N}\xi_\epsilon |u|^p-1\right)_+^{\beta-1}\int_{\mathbb{R}^N}\xi_\epsilon(x)|u|^{p-2}(u-v)\varphi,
\end{aligned}
\end{equation}
and as before, we have that
\begin{equation}\label{3-18}
\begin{aligned}
&\lambda_\epsilon\int_{\mathbb{R}^N}\xi_\epsilon\left(|u|^{p-2}u-|u|^{p-2}v\right)\varphi\\
\leq& \lambda_\epsilon\left(\int_{\mathbb{R}^N}\xi_\epsilon|u|^{p-2}|u-v|^2\right)^{\frac{1}{2}}\left(\int_{\mathbb{R}^N}\xi_\epsilon|u|^{p-2}|\varphi|^2\right)^{\frac{1}{2}}\\
\leq&C\lambda_\epsilon\left(\lambda_\epsilon^{-1}\langle DI_\epsilon(u),u-v\rangle\right)^{\frac{1}{2}}\|u\|_p^{\frac{p-2}{2}}\|\varphi\|_p\\
=&C\lambda_\epsilon^{\frac{1}{2}}\left(\langle DI_\epsilon(u),u-v\rangle\right)^{\frac{1}{2}}\|u\|_p^{\frac{p-2}{2}}\|\varphi\|_p.
\end{aligned}
\end{equation}
Combining \eqref{3-15}, \eqref{3-16}, \eqref{3-17} and \eqref{3-18}, we obtain that
\begin{equation}\label{3-19}
\begin{aligned}
\|DI_\epsilon(u)\|\leq
\left\{\begin{array}{ll}
\|u-v\|^{p-1}+C\lambda_\epsilon^{\frac{1}{p}}\langle DI_\epsilon(u),u-v\rangle^{1-\frac{1}{p}},&\hbox{if } 1<p<2;\\
\|u-v\|\left(\|u\|^{p-2}+\|v\|^{p-2}\right)+C\lambda_\epsilon^{\frac{1}{2}}\langle DI_\epsilon(u),u-v\rangle^{\frac{1}{2}}\|u\|^{\frac{p-2}{2}},&\hbox{if }p\geq 2.
\end{array}\right.
\end{aligned}
\end{equation}
Now we estimate $\|u\|$ and $\lambda_u$.
\begin{equation}\label{3-20}
\begin{aligned}
&I_\epsilon(u)+\frac{1}{r}\int_{\mathbb{R}^N}|\nabla v|^{p-2}\nabla v\nabla u+V(\epsilon x)|v|^{p-2}vu-\frac{1}{r}\int_{\mathbb{R}^N}|\nabla u|^{p}+V(\epsilon x)|u|^{p}\\
=&\left(\frac{1}{p}-\frac{1}{r}\right)\int_{\mathbb{R}^N}|\nabla u|^p+V(\epsilon x)|u|^p+\frac{1}{p\beta}\lambda_\epsilon^{\frac{\beta}{\beta-1}}-\frac{1}{r}\lambda_\epsilon\int_{\mathbb{R}^N}\xi_\epsilon|z|^{p-2}vu\\
&-\int_{\mathbb{R}^N}F_\epsilon(x,u)+\frac{1}{r}\int_{\mathbb{R}^N}f_\epsilon(x,u)u\\
\geq&C\|u\|^p+\left(\frac{1}{p\beta}-\frac{1}{r}\right)\lambda_\epsilon^{\frac{\beta}{\beta-1}}+\frac{1}{r}\lambda_\epsilon\left(\int_{\mathbb{R}^N}\xi_\epsilon|u|^p-1-\int_{\mathbb{R}^N}\xi_\epsilon|z|^{p-2}vu\right)\\
=&C\|u\|^p+\left(\frac{1}{p\beta}-\frac{1}{r}\right)\lambda_\epsilon^{\frac{\beta}{\beta-1}}-\frac{1}{r}\lambda_\epsilon+\frac{1}{r}\lambda_\epsilon\int_{\mathbb{R}^N}\xi_\epsilon u(|u|^{p-2}u-|z|^{p-2}v),
\end{aligned}
\end{equation}
thus
\begin{equation}\label{3-21}
\begin{aligned}
&C\|u\|^p+C\lambda_\epsilon^{\frac{\beta}{\beta-1}}\\
\leq&I_\epsilon(u)+\frac{1}{r}\int_{\mathbb{R}^N}\left(|\nabla v|^{p-2}\nabla v-|\nabla u|^{p-2}\nabla u\right)\nabla u+\frac{1}{r}\int_{\mathbb{R}^N}V(\epsilon x)\left(|v|^{p-2}v-|u|^{p-2}u\right)u\\
&+\frac{1}{r}\lambda_\epsilon-\frac{1}{r}\lambda_\epsilon\int_{\mathbb{R}^N}\xi_\epsilon u(|u|^{p-2}u-|z|^{p-2}v)\\
\leq&I_\epsilon(u)+C\|u-v\|^p+\varepsilon\|u\|^p+C\lambda_\epsilon+C\lambda_\epsilon\left(C\|u-v\|_p^p+\varepsilon\int_{\mathbb{R}^N}\xi_\epsilon|u|^p\right)\\
\leq&I_\epsilon(u)+C\|u-v\|^{p\beta}+\varepsilon\|u\|^p+\varepsilon\lambda_\epsilon^{\frac{\beta}{\beta-1}}+C,
\end{aligned}
\end{equation}
where we have utilized an inequality in \cite{zll1}, that is, for any $\varepsilon>0$ there exists $C_{\varepsilon}>0$ such that for $\xi, \eta \in \mathbb{R}^{N}$
\begin{equation}\label{3-22}
\begin{aligned}
\left||\xi|^{p-2} \xi-|\eta|^{p-2} \eta\right|| \xi | & \leq C\left(|\xi|^{p-1}|\xi-\eta|+|\xi||\xi-\eta|^{p-1}\right) \\
& \leq C_{\varepsilon}|\xi-\eta|^{p}+\varepsilon|\xi|^{p}.
\end{aligned}
\end{equation}
Therefore
\begin{equation}\label{3-23}
\|u\|^p+\lambda_\epsilon^{\frac{\beta}{\beta-1}}\leq CI_\epsilon(u)+C\|u-v\|^{p\beta}+C.
\end{equation}
From \eqref{3-23}, we have
\begin{equation}\label{3-24}
\begin{aligned}
\|u\|\leq&CI_\epsilon^{\frac{1}{p}}(u)+C\|u-v\|^\beta+C,\\
\|v\|\leq&\|u\|+\|u-v\|\leq CI_\epsilon^{\frac{1}{p}}(u)+C\|u-v\|^\beta+C,\\
\lambda_\epsilon\leq&I_\epsilon(u)^{\frac{\beta-1}{\beta}}+C\|u-v\|^{p(\beta-1)}+C.
\end{aligned}
\end{equation}
Combining with \eqref{3-19}, we have that
\begin{equation}\label{3-25}
\begin{aligned}
\|DI_\epsilon(u)\|\leq
\left\{\begin{array}{ll}
\|u-v\|^{p-1}+C\left(I_\epsilon(u)^{\frac{\beta-1}{\beta}}+\|u-v\|^{p(\beta-1)}+1\right)\|u-v\|^{p-1},&\hbox{if } 1<p<2;\\
C\left(I_\epsilon(u)^{\frac{\beta-1}{\beta}+\frac{p-2}{p}}+\|u-v\|^{p\beta\left(\frac{\beta-1}{\beta}+\frac{p-2}{p}\right)}+1\right)^2\|u-v\|,&\hbox{if } p\geq 2.
\end{array}\right.
\end{aligned}
\end{equation}
Thus we obtain there exists $a=a(c_0,b_0)$ such that if $\|DI_\epsilon(u)\|\geq b_0$, $|I_\epsilon(u)|\leq c_0$, $\|u-v\|\geq a$. By \eqref{3-13}, \eqref{3-24} and the choice of $\beta$, we deduce the desired.
\end{proof}
\begin{lemma}\label{le3-4}
$\Gamma_j$ is nonempty.
\end{lemma}
\begin{proof}
Assume $B=\{x\in\mathbb{R}^N\left||x|\leq r_0\right.\}\subset M_\epsilon$. Let $\{e_n\}_{n=1}^\infty$ be a family of linear independent functions in $C_0^\infty(B)$. Then there exists an increasing sequence of $\{R_n\}$ such that
$$
J_0(u)<0,~~\forall u\in E_n, ~~\|u\|\geq R_n,
$$
where $E_n=span\{e_1, e_2, \cdots, e_n\}$ and
$$
J_0(u)=\frac{1}{p}\int_{\mathbb{R}^N}|\nabla u|^p+\tilde{b}|u|^p-\int_{\mathbb{R}^N}F_1(x,u).
$$
Denote $B_n$ the unit ball in $\mathbb{R}^n$, define $\varphi_n\in C(B_n,C_0^\infty(B))$ as
$$
\varphi_n=R_n\sum\limits_{i=1}^nt_ie_i,~~~~t=(t_1,\cdots,t_n)\in B_1.
$$
Let
$$
\Gamma_j=\{E|E\subset W^{1,p}(\mathbb{R}^N), E \hbox{ compact}, -E=E, \gamma(E\cap\eta^{-1}(\Sigma))\geq j \hbox{ for } \eta \in\Lambda\},
$$
$$
\Lambda=\{\eta|\eta\in C(W^{1,p},W^{1,p}), \eta\hbox{ odd}, \eta(P^\delta)\subset P^\delta, \eta(Q^\delta)\subset Q^\delta, \eta(x)=x \hbox{ if } I_\epsilon<0\},
$$
then $\varphi_n(B_n)\subset \Gamma_n$, we refer \cite{llw-jde} for the proof.
\end{proof}
Having verified all the assumptions of Theorem \ref{th3-1}, we have the following existence theorem.
\begin{theorem}\label{th3-2}
Assume $(V_1)$, $(V_2)$. Then the functional $I_\epsilon$, $\epsilon\in(0,1]$, has infinitely many sign-changing critical points, the corresponding critical values are defined as
$$
c_j(\epsilon)=\inf\limits_{A\in\Gamma_j}\sup\limits_{x\in A\setminus O}f(x),~~~~~j=1,2,\cdots.
$$
Moreover
\begin{enumerate}
\item there exists $m_j$, $j=1,\cdots$, independent of $\epsilon$ such that
\begin{equation}\label{3-26}
c_j(\epsilon)\leq m_j,~~~~j=1, 2,\cdots.
\end{equation}
\item If $c_j(\epsilon)=\cdots=c_{j+k}(\epsilon)=c$, then $\gamma(K^*_c)\geq k+1$.
\end{enumerate}
\end{theorem}
\begin{proof}
It remains to verify \eqref{3-26}. $H_n:=\varphi_n(B_n)\subset \Gamma_n$. For $t\in B_n$, $u=\varphi_n(t)$, $\left(\int_{\mathbb{R}^N}\xi_\epsilon|u|^p-1\right)_+=0$, we have
$$
I_\epsilon(u)\leq J_0(u).
$$
Hence
$$
c_j(\epsilon)\leq m_j:=\sup\limits_{H_n} J_0(u).
$$
\end{proof}
\section{Profile decomposition of solution sequences and exclusion of blowing-up}\label{4}
In this section, we prove that the sign-changing solutions obtained in Theorem \ref{th3-2} cannot blow up. More precisely, we will prove the following:
\begin{proposition}\label{prop4-1}
Let $\{u_n\}$ be a sequence of sign-changing critical points obtained in Theorem \ref{th3-2} corresponding to $\{\epsilon_n\}\to 0$. Then there exists $M>0$ and $\epsilon_0$ such that, for $0<\epsilon_n<\epsilon_0$,
$$
\sup\limits_{x\in\mathbb{R}^N}|u_n(x)|<M.
$$
\end{proposition}
To prove it, we need some lemmas.
\begin{lemma}\label{le4-1}
Let $\{u_n\}$ satisfies that $DI_{\epsilon_n}(u_n)=0$, $I_{\epsilon_n}(u_n)\leq M$. Then there exist $\rho$ and $\eta_M>0$ independent of $\epsilon_n$ such that
$$
\rho\leq \|u_n\|\leq \eta_M,~~~~\hbox{ and }~~~~Q_{\epsilon_n}(u_n)\leq \eta_M.
$$
\end{lemma}
\begin{proof}
The existence of $\eta_M$ has been proved by \eqref{2-1} in Lemma \ref{le2-1}. Taking $\varphi=u_n$ in \eqref{1-8}, we obtain that
\begin{align*}
C\|u_n\|^p\leq&\int_{\mathbb{R}^N}|\nabla u_n|^p+V(\epsilon_n x)|u_n|^p+\lambda_{\epsilon_n}\int_{\mathbb{R}^N}\xi_{\epsilon_n}|u_n|^p\\
=&\int_{\mathbb{R}^N}f_{\epsilon_n}(x,u_n)u_n\\
\leq&C(\|u_n\|^{p^*}+\|u_n\|^q).
\end{align*}
Since $u_n\nequiv 0$, it follows that $\|u_n\|\geq\rho$.
\end{proof}
By Lemma \ref{le4-1}, $\{u_n\}$ is bounded in $W^{1,p}(\mathbb{R}^N)$, we can use the following profile decomposition result introduced in \cite{zll}.
\begin{proposition}\label{prop4-2}
There exists two finite sets $\Lambda_1$ and $\Lambda_\infty$, such that
$$
u_{n}=\sum_{k \in \Lambda_{1}} U_{k}\left(\cdot-x_{n, k}\right)+\sum_{k \in \Lambda_{\infty}} \sigma_{n, k}^{\frac{N-p}{p}} U_{k}\left(\sigma_{n, k}\left(\cdot-x_{n, k}\right)\right)+r_{n},
$$
and
\begin{enumerate}
\item If $k\in\Lambda_1,$ then $U_k \in W^{1, p}\left(\mathbb{R}^{N}\right),$ and $V_k=|U_k|$ satisfies the differential inequality
$$
\int_{\mathbb{R}^{N}}|\nabla V_k|^{p-2} \nabla V_k \nabla \varphi d x+a\int_{\mathbb{R}^{N}} V_k^{p-1} \varphi d x \leq \int_{\mathbb{R}^{N}} V_k^{p^{*}-1} \varphi d x+\mu \int_{\mathbb{R}^{N}} V_k^{q-1} \varphi d x,
$$
for $\varphi \geq 0, \varphi \in W^{1, p}\left(\mathbb{R}^{N}\right)$. Besides there exists a constant $C>0$ such that
$$
|U_k|\leq\frac{C}{(1+|x|^{\frac{p}{p-1}})^{\frac{N-p}{p}}},~~~~\hbox{ for } x\in\mathbb{R}^N.
$$
\item If $k\in\Lambda_\infty,$ then $U_k \in \mathcal{D}^{p},$ and $V_k=|U_k|$ satisfies the differential inequality
$$
\int_{\mathbb{R}^{N}}|\nabla V_k|^{p-2} \nabla V_k \nabla \varphi d x \leq \int_{\mathbb{R}^{N}} V_k^{p^{*}-1} \varphi d x,
$$
for  $\varphi \geq 0$, $\varphi \in \mathcal{D}^{p}$. Besides there exists a constant $c>0$ such that
$$
|U_k|\leq c^{-1}e^{-c(1+|x|)},~~~~\hbox{ for } x\in\mathbb{R}^N.
$$
\item For $k \in \Lambda_{1}, u_{n}\left(\cdot+x_{n, k}\right) \rightharpoonup U_{k}$ in $W^{1, p}\left(\mathbb{R}^{N}\right)$.\\
For $k \in \Lambda_{\infty}, \sigma_{n, k}^{-\frac{N-p}{p}} u_{n}\left(\sigma_{n, k}^{-1} \cdot+x_{n, k}\right) \rightharpoonup U_{k}$ in $\mathcal{D}^{p}$.
\item For $k, l \in \Lambda_{1}, k \neq l$
$$
\left|x_{n, k}-x_{n, l}\right| \rightarrow+\infty.
$$
For $k, l \in \Lambda_{\infty}, k \neq l$
$$
\frac{\sigma_{n, k}}{\sigma_{n, l}}+\frac{\sigma_{n, l}}{\sigma_{n, k}}+\sigma_{n, k} \sigma_{n, l}\left|x_{n, k}-x_{n, l}\right|^{2} \rightarrow+\infty.
$$
\item $
r_{n}=u_{n}-\sum_{k \in \Lambda_{1}} U_{k}\left(\cdot-x_{n, k}\right)-\sum_{k \in \Lambda_{\infty}} \sigma_{n, k}^{\frac{N-p}{p}} U_{k}\left(\sigma_{n, k}\left(\cdot-x_{n, k}\right)\right) \rightarrow 0 \text { in } L^{p^{*}}\left(\mathbb{R}^{N}\right).
$
\end{enumerate}
\end{proposition}
\begin{lemma}\label{le4-2}
Assume $DI_\epsilon(u)=0$, $I_\epsilon(u)\leq M$. Then for any $\delta>0$, there exists $C=C(\delta,M)$ such that $|u(x)|\leq C\epsilon^{\frac{\gamma}{p}}$ for $x\in \mathbb{R}^N\setminus (M_\epsilon)^\delta$.
\end{lemma}
\begin{proof}
For any $\delta>0$ and $x_0\in\mathbb{R}^N\setminus(\mathcal{M}_\epsilon)^\delta$, $0<r<R\leq \frac{\delta}{2}$. Choose $\phi\in C_0^\infty(\mathbb{R}^N)$ such that $0\leq\phi\leq 1$, $\phi(x)=1$ for $x\in B_r=B_r(x_0)$; $\phi(x)=0$ for $x\notin B_R=B_R(x_0)$ and $|\nabla \phi|\leq\frac{C}{R-r}$, denote $B=B_{\frac{\delta}{2}}$.

Taking $\varphi=u|u_T|^{p(s-1)}\phi^p$ as a test function in \eqref{1-8}, where $s>1$ is to be chosen
\begin{equation}\label{add4-1}
\begin{aligned}
&\int_{B}|\nabla u|^{p-2}\nabla u\nabla\varphi\\
\geq&\int_{|u|\geq T}|\nabla u|^{p}|u_T|^{p(s-1)}\phi^p+\int_{|u|\leq T}|\nabla u|^{p-2}\nabla u\nabla\left(u|u_T|^{p(s-1)}\right)\phi^p\\
&+\int_{B}|\nabla u|^{p-2}\nabla uu|u_T|^{p(s-1)}p\phi^{p-1}\nabla\phi\\
\geq&\int_{|u|\geq T}|\nabla(u|u_T|^{(s-1)})|^p\phi^p+\frac{p(s-1)+1}{s^p}\int_{|u|\leq T}|\nabla (u|u_T|^{(s-1)})|^p\phi^p\\
&+\int_{B}|\nabla u|^{p-2}\nabla uu|u_T|^{p(s-1)}p\phi^{p-1}\nabla\phi\\
\geq&\frac{S_p}{s^p}\int_{B}\left|\nabla(u|u_T|^{(s-1)}\phi)\right|^p-C\int_{B}\left|u|u_T|^{(s-1)}\right|^p|\nabla\phi|^p\\
\geq&\frac{C}{s^p}\left(\int_{B}\left|u|u_T|^{(s-1)}\phi\right|^{p^*}\right)^{\frac{p}{p^*}}-\frac{C}{(R-r)^p}\int_{B_R}\left|u|u_T|^{(s-1)}\right|^p|\nabla \phi|^p.
\end{aligned}
\end{equation}
On the other hand, by Proposition \ref{prop2-1}, and the fact that $Q_{\epsilon}(u)\leq \eta_M$, we have
\begin{equation}\label{add4-2}
\begin{aligned}
&\int_{B}f_\epsilon(x,u)\varphi\leq C\epsilon^{-(p^*-r)}\int_{B}|u|^r|u_T|^{p(s-1)}\phi^p\\
\leq&C\epsilon^{-(p^*-r)}\left(\int_{B}|u|^r\right)^{1-\frac{p}{r}}\left(\int_{B}|u|u_T|^{(s-1)}\phi|^{r}\right)^{\frac{p}{r}}\\
\leq&C\epsilon^{-(p^*-r)}\left(\int_{B}|u|^p\right)^{\left(1-\frac{p}{r}\right)t}\left(\int_{B}|u|^{p^*}\right)^{\left(1-\frac{p}{r}\right)(1-t)}\left(\int_{B}|u|u_T|^{(s-1)}\phi|^{r}\right)^{\frac{p}{r}}\\
\leq& C\epsilon^{-(p^*-r)}\epsilon^{\gamma\left(1-\frac{p}{r}\right)t}\left(\int_{B}|u|u_T|^{(s-1)}\phi|^{r}\right)^{\frac{p}{r}},
\end{aligned}
\end{equation}
where $t=\frac{p^*-r}{p^*-p}$. Since $\gamma>\frac{p^2r}{(r-p)(N-p)}$, we have $\gamma\left(1-\frac{p}{r}\right)t-(p^*-r)>0$. Considering \eqref{add4-2}, we obtain that
\begin{equation}\label{add4-3}
\begin{aligned}
\left(\int_{B_r}|u|u_T|^{(s-1)}|^{p^*}\right)^{\frac{p}{p^*}}\leq\frac{Cs^p}{(R-r)^p}\left(\int_{B_R}|u|u_T|^{(s-1)}|^r\right)^{\frac{p}{r}},
\end{aligned}
\end{equation}
let $T\to\infty$, we deduce
$$
\left(\int_{B_r}|u|^{p^*s}\right)^{\frac{1}{p^*s}}\leq\left(\frac{Cs}{(R-r)}\right)^{\frac{1}{s}}\left(\int_{B_R}|u|^{rs}\right)^{\frac{1}{rs}}.
$$
Let $\chi=\frac{p^*}{r}>1$, $s=\chi^i$, $r_i=r+\frac{1}{2^i}(R-r)$, $i=0, 1,\cdots$, we obtain
$$
\left(\int_{B_{r_i}}|u|^{r\chi^{i+1}}\right)^{\frac{1}{r\chi^{i+1}}}\leq\left(\frac{C2^i\chi^i}{R-r}\right)^{\frac{1}{\chi^i}}\left(\int_{B_{r_{i-1}}}|u|^{r\chi^i}\right)^{\frac{1}{r\chi^i}},~~~~i=1, 2,\cdots,
$$
 by iteration, we have
\begin{equation}\label{add4-4}
\|u\|_{L^\infty(B_r)}\leq\frac{C}{(R-r)^t}\|u\|_{L^{p^*}}(B_R),
\end{equation}
where $0<r<R\leq R_0<\delta/2$ and $t=1/(\chi-1)$. Thus
$$
\|u\|_{L^\infty(B_r)}\leq\frac{C}{(R-r)^t}\|u\|^{\frac{p^*-p}{p^*}}_{L^\infty(B_R)}\|u\|^{\frac{p}{p^*}}_{L^p(B_R)}\leq\frac{1}{2}\|u\|_{L^\infty(B_R)}+\frac{C}{(R-r)^{t'}}\|u\|_{L^p(B_R)}.
$$
By an iteration argument, we obtain
$$
\|u\|_{L^\infty(B_r)}\leq\frac{C}{(R-r)^{t'}}\|u\|_{L^p(B_R)},~~~~0<r<R\leq R_0<\frac{\delta}{2}.
$$
Note that $\|u\|_{L^p(B_R)}\leq C\epsilon^{\frac{\gamma}{p}}$ and $C$ is dependent of the choice of $x_0$, we deduce the desired.
\end{proof}
\begin{lemma}\label{le4-3}
Let $\{u_n\}$ be a sequence of sign-changing critical points obtained in Theorem \ref{th3-2} corresponding to $\{\epsilon_n\}\to 0$. For any $\delta>0$, there exists $\epsilon_0$, such that if $\epsilon_n<\epsilon_0$, then up to a subsequence,  $x_{n,k}\in (M_{\epsilon_n})^\delta$ if $k\in\Lambda_1$ and $x_{n,k}\in(M_{\epsilon_n})^{\sigma_{n,k}^{-1/p}}$ if $k\in\Lambda_\infty$.
\end{lemma}
\begin{proof}
For simplicity, we denote $|u_n|$ by $w_n$.

Since $u_n$ satisfies \eqref{1-8}, by Proposition \ref{prop2-1}, we deduce that there exists $C>0$ independent of $n$ such that $w_n$ satisfies
$$
-\Delta_p w_n\leq C\epsilon_n^{-(p^*-r)}w_n^{r-1}.
$$
Let $v_n=\sigma_{n,k}^{-\frac{N-p}{p}}w_n(\sigma_{n,k}^{-1}x+x_{n,k})$, then $v_n$ satisfies
$$
-\Delta_p v_n\leq\sigma_n^{\frac{r (N-p)}{p}-N}\epsilon_n^{-(p^*-r)}v_n^{r-1}.
$$
Let $\omega_n=\sigma_{n,k}^{-\frac{N-p}{p}}u_n(\sigma_{n,k}^{-1}\cdot+x_{n,k})$, then it satisfies
\begin{equation}\label{ad4-1}
\begin{aligned}
&-\Delta_p\omega_n+\sigma_{n,k}^{-p}V(\epsilon_n\sigma_{n,k}^{-1}x+\epsilon_nx_{n.k})\omega_n+\sigma_{n,k}^{-p}\lambda_n\xi_{\epsilon_n}(\sigma_{n.k}^{-1}x+x_{n,k})\omega_n\\
=&\sigma_{n,k}^{\frac{N-p}{p}-N}f_{\epsilon_n}(\sigma_{n,k}^{-1}x+x_{n,k},\sigma_{n,k}^{(N-p)/p}\omega_n),
\end{aligned}
\end{equation}
where
$$
\lambda_n=\left(\int_{\mathbb{R}^N}\xi_{\epsilon_n}|u_n|^p-1\right)_+^{\beta-1}.
$$
Let
$$
\Lambda_{n,k}=\{\sigma_{n,k}(y-x_{n,k})|y\in M_{\epsilon_n}\},
$$
then
$$
\sigma_{n,k}^{-1}x+x_{n,k}\in M_{\epsilon_n}\Leftrightarrow x\in \Lambda_{n,k},
$$
and
$$
dist(\sigma_{n,k}^{-1}x+x_{n,k},M_{\epsilon_n})=\sigma_{n,k}^{-1}dist(x,\Lambda_{n,k}).
$$
It follows that
\begin{equation}\label{4-1}
\xi_{\epsilon_{n}}\left(\sigma_{n,k}^{-1} x+x_{n,k}\right)=\left\{\begin{array}{ll}{0,} & {\text { if } x \in \Lambda_{n,k}},\\
 {\epsilon_{n}^{-\gamma} \zeta\left(\sigma_{n,k}^{-1} \operatorname{dist}\left(x, \Lambda_{n,k}\right)\right),} & {\text { if } x \notin \Lambda_{n,k}}.
 \end{array}\right.
\end{equation}
Without loss of generality, we assume that $\{x_{n,k}\}_n$ satisfies either $\{x_{n,k}\}_n\subset (\mathcal{M}_{\epsilon_n})^\delta$ or $\{x_{n,k}\}_n\subset\mathbb{R}^N\setminus (M_{\epsilon_n})^\delta$.

We first prove that for $k\in\Lambda_\infty$, $\{x_{n,k}\}\subset (\mathcal{M}_{\epsilon_{n}})^{\sigma_{n,k}^{-1/p}}$. On the contrary, if there is some $i_0\in\Lambda_\infty$ such that $\{x_{n,i_0}\}\subset\mathbb{R}^N\setminus(\mathcal{M}_{\epsilon_{n}})^{\sigma_{n,i_0}^{-1/p}}$. For convenience, we denote $x_{n,i_0}$, $\sigma_{n,i_0}$ and $\Lambda_{n,i_0}$ by $x_n$, $\sigma_n$ and $\Lambda_n$, respectively.

We will deduce a contradiction in this case. In fact, if this happens, then
$$
\lim _{n \rightarrow \infty} \sigma_{n} \operatorname{dist}\left(x_{n}, \partial \mathcal{M}_{\epsilon_{n}}\right)=+\infty,
$$
which implies that for any $R>0$, when $n$ is large enough,
$$
B_R(0)\subset\mathbb{R}^N\setminus (\Lambda_n)^1.
$$
Since $U_{i_0}\nequiv 0$, there exists $R_0>0$ such that
\begin{equation}\label{4-9}
b':=\int_{B_{R_0}(0)}|U_{i_0}|^p>0.
\end{equation}
It follows that
$$
v_{n}=\sigma_{n}^{-\frac{N-p}{p}}\left|u_{n}\right|\left(\sigma_{n}^{-1} \cdot+x_{n}\right) \rightharpoonup\left|U_{i_0}\right| \neq 0 \text { in } \mathcal{D}^{p}.
$$
Then by \eqref{4-9}, we get that when $n$ large enough
\begin{equation}\label{4-11}
\frac{1}{2} b^{\prime} \leq \quad \int_{B_{R_0}(0)} v_{n}^{p}.
\end{equation}
From $Q_{\epsilon_n}(u_n)\leq \eta_M$ and the definition of $\xi_{\epsilon_n}$, we get that, there exists $C=C(k)>0$ such that
\begin{equation}\label{4-12}
\int_{\mathbb{R}^{N} \backslash \mathcal{M}_{\epsilon_{n}}} \zeta\left(\operatorname{dist}\left(x, \mathcal{M}_{\epsilon_{n}}\right)\right) w_{n}^{p} d x \leq C \epsilon_{n}^{\gamma}.
\end{equation}
It follows that
\begin{equation}\label{4-13}
\int_{\mathbb{R}^{N} \backslash \Lambda_{n}} \zeta\left(\sigma_{n}^{-1} \operatorname{dist}\left(x, \Lambda_{n}\right)\right) v_{n}^{p} d x \leq C \epsilon_{n}^{\gamma} \sigma_{n}^{p}.
\end{equation}
By definition of $\zeta$, we get that when $n$ is large enough, for $x\in B_R(0)$,
$$
\zeta(\sigma_n^{-1}dist(x,\Lambda_n))\geq C\sigma_n^{-a},
$$
thus, we get that
$$
C \sigma_{n}^{-a} \quad \int_{B_{R}(0)} v_{n}^{p} \leq \int_{B_{R}(0)} \zeta\left(\sigma_{n}^{-1} \operatorname{dist}\left(x, \Lambda_{n}\right)\right) v_{n}^{p} d x \leq C \epsilon_{n}^{\gamma} \sigma_{n}^{p},
$$
that is
$$
\int_{B_{R}(0)} v_{n}^{p} \leq C \epsilon_{n}^{\gamma} \sigma_{n}^{p+a}.
$$
Combining with \eqref{4-11} that
\begin{equation}\label{4-10}
\epsilon^{-1}\leq\sigma_n^{\frac{p+a}{\gamma}}.
\end{equation}
Thus
$$
-\Delta_p v_n\leq\sigma_n^{\frac{r(N-p)}{p}-N+\left(\frac{p+a}{\gamma}\right)(p^*-r)}v_n^{r-1}.
$$
Since $\gamma>\frac{(p+a)p}{N-p}$, we get that $\frac{r(N-p)}{p}-N+\left(\frac{p+a}{\gamma}\right)(p^*-r)<0$. Let $n\to\infty$, we have by $v_n\rightharpoonup W=|U_{i_0}|$ in $\mathcal{D}^{p}$ that
$$
-\Delta_p W\leq 0~~\hbox{ in } \mathbb{R}^N.
$$
This contradicts $W\geq 0$, $W\nequiv 0$ and $\int_{\mathbb{R}^N}W^{2^*}<\infty$.

Thus, for all $k\in\Lambda_\infty$, $\{x_{n,k}\}\subset(\mathcal{M}_{\epsilon_n})^{\sigma_{n,k}^{-1/p}}$. Now we prove that $\{x_{n,k}\}\subset \mathcal{M}_{\epsilon_n}^\delta$, for any $\delta>0$ and $k\in\Lambda_1$.

For any $\delta>0$, there exists $N$ such that for $n>N$, $\sigma_{n,k}^{-\frac{1}{p}}<\delta$. We deduce that $\{x_{n,k}\}_{k\in\Lambda_1}\subset (\mathcal{M}_{\epsilon_n})^{2\delta}$. In fact, if else, we may assume that up to a subsequence, there is some $i\in\Lambda_1$ such that $\{x_{n,i}\}\in\mathbb{R}^N\setminus(\mathcal{M}_{\epsilon_n})^{2\delta}$ for all $n$. Then we can take $R$ large enough such that $B_R(x_{n,i})\subset\mathbb{R}^N\setminus(\mathcal{M}_{\epsilon_n})^{2\delta}$ and
\begin{equation}\label{addd4-1}
\int_{B_R(x_{n,i})}|U_i(x-x_{n,i})|^{p^*}>C.
\end{equation}
In fact, if there is some $x_0\in B_R(x_{n,i})\cap(\mathcal{M}_{\epsilon_n})^{2\delta}$, by definition, there exists some $y_0$ such that $\epsilon_ny_0\in\mathcal{M}$ and $dist(x_0,y_0)<2\delta$, thus $dist(x_{n,i},y_0)\leq R+2\delta$, and $dist(\epsilon_n x_{n,i}, \mathcal{M})\leq dist(\epsilon_n x_{n,i},\epsilon_n y_0)=\epsilon_n  dist(x_{n,i}, y_0)\leq\epsilon_n (R+2\delta)$.

On the other hand, since $\{x_{n,k}\}\subset(\mathcal{M}_{\epsilon_n})^\delta$ for all $k\in\Lambda_\infty$, by Proposition \ref{prop4-2}, we have that
\begin{align*}
&\int_{B_{R}(x_{n,i})}\left|\sigma^{\frac{N-p}{p}}_{n,k}U_k(\sigma_{n,k}(x-x_{n,k}))\right|^{p^*}\\
\leq&\sigma^N_{n,k}\int_{|x-x_{n,k}|\geq\delta}\frac{1}{\left(1+\sigma_{n,k}|x-x_{n,k}|\right)^{\frac{Np}{p-1}}}=O(\sigma_{n,k}^{-\frac{N}{p-1}}).
\end{align*}
Choose $n$ large enough, such that $|x_{n,k}-x_{n,i}|>>1$, for $k\in \Lambda_1$ and $k\neq i$. By  Lemma \ref{le4-2} and the fact that $\{U_k\}_{k\in\Lambda_1}$ decay exponentially, we deduce that
$$
\int_{B_{R}(x_{n,i})}|u_n|^{p^*}\leq C_{\delta,R}\epsilon_n^{\frac{N\gamma}{N-p}},
$$
$$
\int_{B_{R}(x_{n,i})}|U_k(x-x_{n,k})|^{p^*}=o(1),
$$
combining above, we deduce that
\begin{align*}
&\int_{B_{R}(x_{n,i})}|U_i(x-x_{n,i})|^{p^*}\\
\leq& C\left(\int_{B_{R}(x_{n,i})}|u_n|^{p^*}
+\sum\limits_{k\in\Lambda_1\setminus{i}}|U_k(x-x_{n,k})|^{p^*}+\sum\limits_{k\in\Lambda_\infty}\left|\sigma_{n,k}^{\frac{N-p}{p}}U_k(\sigma_{n,k}(x-x_{n,k}))\right|^{p^*}\right)+o(1)\\
\leq& C(\epsilon_n^{\frac{N\gamma}{N-p}}+\max\limits_{k\in\Lambda_\infty}\{\sigma_{n,k}^{-\frac{N}{p-1}}\})+o(1),
\end{align*}
which is a contradiction with \eqref{addd4-1}. Thus $\{x_{n,k}\}_{k\in\Lambda_1}\subset(\mathcal{M}_{\epsilon_n})^{2\delta}$.
\end{proof}
Let $i_\infty\in\Lambda_\infty$ be such that
$$
\sigma_n:=\sigma_{n,i_\infty}=\min\{\sigma_{n,k}|k\in\Lambda_\infty\}.
$$
Denote
$$
x_n:=x_{n,i_\infty}.
$$
By the proof of \cite{ds1}, we can find a constant $\bar{C}>0$ such that, up to a subsequence, the set
$$
\mathcal{A}_n^1=B_{(\bar{C}+5)\sigma_n^{-\frac{1}{p}}}(x_n)\setminus B_{\bar{C}\sigma_n^{-\frac{1}{p}}}(x_n)
$$
satisfies that
$$
\mathcal{A}_n^1\cap\{x_{n,k}|k\in\Lambda_\infty\}=\emptyset.
$$
We denote
$$
\mathcal{A}_n^2=B_{(\bar{C}+4)\sigma_n^{-\frac{1}{p}}}(x_n)\setminus B_{(\bar{C}+1)\sigma_n^{-\frac{1}{p}}}(x_n),
$$
and
$$
\mathcal{A}_n^3=B_{(\bar{C}+3)\sigma_n^{-\frac{1}{p}}}(x_n)\setminus B_{(\bar{C}+2)\sigma_n^{-\frac{1}{p}}}(x_n).
$$
Then we have the following estimate on $\mathcal{A}_n^2$.
\begin{lemma}\label{adle4-1}
Denote $\{u_n\}$ to be a sequence of solutions as in Lemma \ref{le4-3}. Then there exists a constant $C>0$, independent of $n$, such that
$$
|u_n(x)|\leq C,~~x\in\mathcal{A}_n^2.
$$
\end{lemma}
\begin{proof}Let $w_n=|u_n|$, then
$$
-\Delta_p w_n+V(\epsilon_n x)w_n+\lambda_n\xi_{\epsilon_n}w_n^{p-1}=f_{\epsilon_n}(x,w_n),
$$
thus
$$
-\Delta_p w_n\leq Cw_n^{p^*-1}.
$$
For $k\in\Lambda_1$
$$
-\Delta_p|U_k|+|U_k|^{p-1}\leq |U_k|^{p^*-1}.
$$
For $k\in\Lambda_\infty$
$$
-\Delta_p|U_k|\leq |U_k|^{p^*-1}.
$$
By Proposition 3.2, Lemma A.6, Lemma A.8, Corollary A.4 and Proposition A.5 in \cite{zll}, we deduce that $w_n(x)\leq C$ for $x\in\mathcal{A}_n^2$ as Proposition 4.1 in \cite{zll}.
\end{proof}
\begin{lemma}\label{adle4-2}
It holds that
$$
\int_{\mathcal{A}_n^3}|\nabla u_n|^p\leq C\sigma_n^{-\frac{N-p}{p}}.
$$
\end{lemma}
\begin{proof}
The proof is similar to Corollary 4.3 in \cite{zll}, thus we omit it here.
\end{proof}

\begin{lemma}\label{le4-4}
Let $B_n=B_{(\bar{C}+2.7)\sigma_n^{-\frac{1}{p}}}(x_n)$, and $\phi\geq 0$ in $\mathbb{R}$, $\phi(t)=0$ if $t\geq(\bar{C}+2.7)\sigma_n^{-\frac{1}{p}}$, $\phi(t)=1$ for $t\leq (\bar{C}+2.1)\sigma_n^{-\frac{1}{p}}$, $\phi'(t)\leq 0$ and $|\phi'(t)|\leq C\sigma_n^{\frac{1}{p}}$ for all $t$. Let $\varphi(x)=\phi(|x-x_n|)$. Then there exists $C>0$ independent of $n$ such that
\begin{equation}\label{ad4-10}
\lambda_n\int_{B_n}|\nabla \xi_{\epsilon_n}||u_n|^p\varphi\leq C\sigma_n^{-\frac{N}{p}+\frac{p+1}{p}},
\end{equation}
and
\begin{equation}\label{ad4-11}
\lambda_n\int_{B_n}\xi_{\epsilon_n}|u_n|^p\varphi\leq C\sigma_n^{-\frac{N-p}{p}}.
\end{equation}
\end{lemma}
\begin{proof}
For all $t\in\mathbb{R}^N$, multiplying both side of \eqref{1-9}  by $(t\cdot\nabla u_n)\varphi$ and integrating in $B_n$, we get that
\begin{equation}\label{4-14}
\begin{aligned}
&\frac{\epsilon_n}{p}\int_{B_n}(t,\nabla V(\epsilon_n x))|u_n|^p\varphi+\frac{\lambda_n}{p}\int_{B_n}(t,\nabla \xi_{\epsilon_n})|u_n|^p\varphi-\int_{B_n}\left(t,\nabla_xF_{\epsilon_n}(x,u_n)\right)\varphi\\
=&\int_{D_n}|\nabla u_n|^{p-2}(t,\nabla u_n)(\nabla u_n,\nabla \varphi)-\frac{1}{p}\int_{D_n}|\nabla u_n|^p(t,\nabla \varphi)\\
&-\frac{1}{p}\int_{D_n}\left(V(\epsilon_n x)+\lambda_n\xi_{\epsilon_n}\right)|u_n|^p(t,\nabla \varphi)+\int_{D_n}F_{\epsilon_n}(x,u_n)(t,\nabla \varphi),
\end{aligned}
\end{equation}
where $$
D_n=\left\{x|(\bar{C}+2.7)\sigma_n^{-\frac{1}{p}}\geq|x-x_n|\geq(\bar{C}+2.1)\sigma_n^{-\frac{1}{p}}\right\}.
$$
 By Lemma \ref{le4-3}, $x_n\in(\mathcal{M}_{\epsilon_n})^{\sigma_n^{-\frac{1}{p}}}$, therefore, $x^*=\lim\limits_{n\to\infty}\epsilon_nx_n\in\overline{\mathcal{M}}$. If $B_n\subset\mathcal{M}_{\epsilon_n}$, then $\xi_{\epsilon_n}=0$ and $\nabla\xi_{\epsilon_n}=0$ in $B_n$, the inequalities hold naturally. Thus we may assume $B_n\cap(\mathcal{M}_{\epsilon_n})^C\neq\emptyset$, in this case, we notice that $x^*\in\partial \mathcal{M}$, taking into consideration that $\mathcal{A}\subset\mathcal{M}$, $\mathcal{A}$ is compact, we have $|\nabla  V(x^*)|> 0$. Thus for all $x\in B_n$, for $n$ large, we obtain
 \begin{equation}\label{4-15}
 \nabla V(\epsilon_n x)\cdot\nabla V(x^*)\geq\frac{1}{2}|\nabla V(x^*)|^2>0.
 \end{equation}
Besides, since
 \begin{equation}\label{4-16}
\nabla \xi_{\epsilon_{n}}(x)=\left\{\begin{array}{ll}
{0,} & {\text { if } x \in \mathcal{M}_{\epsilon_{n}}}, \\ {\epsilon_{n}^{-\gamma} \zeta^{\prime}\left(\operatorname{dist}\left(x, \mathcal{M}_{\epsilon_{n}}\right)\right) \nabla \operatorname{dist}\left(x, \mathcal{M}_{\epsilon_{n}}\right),} & {\text { if } x \in \mathbb{R}^{N} \backslash \mathcal{M}_{\epsilon_{n}}},
\end{array}\right.
\end{equation}
 we have, by assumption $(V_2)$, that for all $x\in B_n\setminus\mathcal{M}_{\epsilon_n}$,
 \begin{equation}\label{4-17}
 \nabla V(x^*)\cdot\frac{\nabla\xi_{\epsilon_n}(x)}{|\nabla\xi_{\epsilon_n}(x)|}=\nabla V(x^*)\cdot\frac{\nabla dist(x,\mathcal{M}_{\epsilon_n})}{|\nabla dist(x,\mathcal{M}_{\epsilon_n})|}\geq c_0>0.
 \end{equation}
 Moreover, by Proposition \ref{prop2-1}, we have for $x\in B_n$,
 \begin{equation}\label{4-18}
 -\int_{B_n}\left(\nabla V(x^*),\nabla_xF_{\epsilon_n}(x,u_n)\right)\varphi\geq 0.
 \end{equation}
 Thus taking $t=\nabla V(x*)$ in \eqref{4-14}, we obtain
 \begin{equation}\label{4-19}
\begin{aligned}
& \frac{c_0}{2p}\lambda_{n} \int_{B_{n}}\left|\nabla \xi_{\epsilon_{n}}\right| |u_{n}|^{p} \varphi \\
\leq & \frac{\lambda_{n}}{p} \int_{B_{n}}\left(\nabla V(x^*) \cdot \nabla \xi_{\epsilon_{n}}\right) |u_{n}|^{p} \varphi\\
\leq& \frac{\epsilon_{n}}{p} \int_{B_{n}}\left(\nabla V(x^*) \cdot\nabla V(\epsilon_{n} x)\right) |u_{n}|^{p} \varphi+\frac{\lambda_{n}}{p} \int_{B_{n}}\left(\nabla V(x^*) \cdot \nabla \xi_{\epsilon_{n}}\right) |u_{n}|^{p} \varphi\\
\leq&\int_{D_{n}}|\nabla u_n|^{p-2}\left(\nabla u_{n} \cdot \nabla \varphi\right)\left(\nabla V(x^*) \cdot \nabla u_{n}\right)-\frac{1}{p} \int_{D_{n}}\left|\nabla u_{n}\right|^{p}(\nabla V(x^*)  \cdot \nabla \varphi)\\
&+\int_{D_{n}} F_{\epsilon_{n}}(x,u_n)(\nabla V(x^*)  \cdot \nabla \varphi)
-\frac{1}{p} \int_{D_{n}}\left(V\left(\epsilon_{n} x\right)+\lambda_{n} \xi_{\epsilon_{n}}\right)(\nabla V(x^*) \cdot \nabla \varphi) |u_{n}|^{p}\\
\leq&C\sigma_n^{\frac{1}{p}}\left(\int_{D_n}|\nabla u_n|^p+F_{\epsilon_n}(x,u_n)+\frac{1}{p}\left(V(\epsilon_n x)+\lambda_n\xi_{\epsilon_n}\right)|u_n|^p\right).
\end{aligned}
\end{equation}

On the other hand, let $\psi \in C_{0}^{\infty}\left(\mathbb{R}^{N}\right)$ be such that $\psi \geq 0$ in $\mathbb{R}^{N}, \psi(x)=1$ for $(\bar{C}+2.7) \sigma_{n}^{-\frac{1}{p}} \geq\left|x-x_{n}\right| \geq$
$(\bar{C}+2.1) \sigma_{n}^{-\frac{1}{p}}, \psi(x)=0$ for $\left|x-x_{n}\right| \geq(\bar{C}+3) \sigma_{n}^{-\frac{1}{p}}$ or $\left|x-x_{n}\right| \leq(\bar{C}+2) \sigma_{n}^{-\frac{1}{p}}$ and
$|\nabla \psi| \leq C \sigma_{n}^{\frac{1}{p}}$. Multiplying both sides of \eqref{1-9} by $u_{n} \psi$ and integrating in $B_{n},$ we get
that
$$
\int_{\mathbb{R}^{N}} \psi\left|\nabla u_{n}\right|^{p}+\int_{\mathbb{R}^{N}} u_{n}|\nabla u_n|^{p-2}\left(\nabla u_{n} \cdot \nabla \psi\right)+\int_{\mathbb{R}^{N}}\left(V\left(\epsilon_{n} x\right)+\lambda_{n} \xi_{\epsilon_{n}}\right)|u_{n}|^{p} \psi=\int_{\mathbb{R}^{N}} f_{\epsilon_{n}}\left(x,u_{n}\right) u_{n} \psi.
$$
This together with Lemma \ref{adle4-1}, Lemma \ref{adle4-2} implies that
\begin{equation}\label{4-20}
\begin{aligned}
 &\lambda_{n} \int_{D_n} \xi_{\epsilon_{n}} |u_{n}|^{p}\\
 \leq & \int_{\mathcal{A}_{n}^{3}} V\left(\epsilon_{n} x\right) |u_{n}|^{p}+\left|\nabla u_{n}\right|^{p} +\left|u_{n}\right|^{p}|\nabla \psi|^{p}+\left|\nabla u_{n}\right|^{p}+\left|f_{\epsilon_{n}}\left(x,u_{n}\right) u_{n}\right| \\
  \leq& C \sigma_{n}^{-(N-p) / p}.
\end{aligned}
\end{equation}
By \eqref{4-19} and \eqref{4-20}, we get \eqref{ad4-10}.

From Lemma \ref{le4-3}, it follows that for $x\in B_n\setminus\mathcal{M}_{\epsilon_n}$,
$$
dist\left(x,\mathcal{M}_{\epsilon_n}\right)\leq C\sigma_n^{-\frac{1}{p}},
$$
thus
\begin{equation}\label{4-21}
\begin{aligned} \chi_{\epsilon_{n}}(x) &=\epsilon_{n}^{-\gamma} \zeta\left(\operatorname{dist}\left(x, \mathcal{M}_{\epsilon_{n}}\right)\right) \\ & \leq C \epsilon_{n}^{-\gamma} \operatorname{dist}\left(x, \mathcal{M}_{\epsilon_{n}}\right) \zeta^{\prime}\left(\operatorname{dist}\left(x, \mathcal{M}_{\epsilon_{n}}\right)\right) \\ &=C \epsilon_{n}^{-\gamma} \operatorname{dist}\left(x, \mathcal{M}_{\epsilon_{n}}\right) \zeta^{\prime}\left(\operatorname{dist}\left(x, \mathcal{M}_{\epsilon_{n}}\right)\right)\left|\nabla \operatorname{dist}\left(x, \mathcal{M}_{\epsilon_{n}}\right)\right| \\ &=C \operatorname{dist}\left(x, \mathcal{M}_{\epsilon_{n}}\right)\left|\nabla \chi_{\epsilon_{n}}(x)\right| \\ & \leq C \sigma_{n}^{-\frac{1}{p}}\left|\nabla \chi_{\epsilon_{n}}(x)\right|. \end{aligned}
\end{equation}
By \eqref{4-21} and \eqref{ad4-10}, we get \eqref{ad4-11}.
\end{proof}
We have the following Pohozaev identity:
\begin{lemma}\label{le4-5}
\begin{equation}\label{4-22}
\begin{aligned}
&\int_{B_n}\left(NF_{\epsilon_n}(x,u_n)-\frac{N-p}{p}u_nf_{\epsilon_n}(x,u_n)\right)\varphi+\int_{B_n}\left(x-x_n,\nabla_xF_{\epsilon_n}(x,u_n)\right)\varphi\\
&-\int_{B_n}\left(V(\epsilon_n x)+\lambda_n\xi_{\epsilon_n}\right)|u_n|^p\varphi-\frac{\epsilon_n}{p}\int_{B_n}\left(x-x_n,\nabla V(\epsilon_n x)\right)|u_n|^p\varphi\\
&-\frac{\lambda_n}{p}\int_{B_n}\left(x-x_n,\nabla\xi_{\epsilon_n}\right)|u_n|^p\varphi\\
=&\frac{1}{p}\int_{B_n}|\nabla u_n|^p\left(x-x_n,\nabla\varphi\right)-\int_{B_n}|\nabla u_n|^{p-2}\left(x-x_n,\nabla u_n\right)\left(\nabla u_n,\nabla\varphi\right)\\
&-\frac{N-p}{p}\int_{B_n}|\nabla u_n|^{p-2}u_n\left(\nabla u_n,\nabla\varphi\right)+\frac{1}{p}\int_{B_n}\left(V(\epsilon_n x)+\lambda_n\xi_{\epsilon_n}\right)\left(x-x_n,\nabla \varphi\right)|u_n|^p\\
&-\int_{B_n}F_{\epsilon_n}(x,u_n)\left(x-x_n,\nabla \varphi\right),
\end{aligned}
\end{equation}
where $B_n$ and $\varphi$ are defined as in Lemma \ref{le4-4}.
\end{lemma}
\begin{proof} The proof is very similar to that of Lemma 5.1 in \cite{zll}, we omit it here.
\end{proof}
\begin{proposition}\label{prop4-3}
If $\max\{p^*-1,p\}<q<p^*$ holds, then $\Lambda_\infty=\emptyset$.
\end{proposition}
\begin{proof}
Since $|\nabla\varphi|\leq C\sigma_n^{\frac{1}{p}}$, $|x-x_n|\leq(\bar{C}+3)\sigma_n^{-\frac{1}{p}}$ and $(x-x_n,\nabla\varphi)\leq 0$ for $x\in\mathcal{A}_n^3$, the right hand of \eqref{4-22} can be estimated as
\begin{equation}\label{4-23}
\begin{aligned}
&\frac{1}{p}\int_{B_n}|\nabla u_n|^p\left(x-x_n,\nabla\varphi\right)-\int_{B_n}|\nabla u_n|^{p-2}\left(x-x_n,\nabla u_n\right)\left(\nabla u_n,\nabla\varphi\right)\\
&-\frac{N-p}{p}\int_{B_n}|\nabla u_n|^{p-2}u_n\left(\nabla u_n,\nabla\varphi\right)+\frac{1}{p}\int_{B_n}\left(V(\epsilon_n x)+\lambda_n\xi_{\epsilon_n}\right)\left(x-x_n,\nabla \varphi\right)|u_n|^p\\
&-\int_{B_n}F_{\epsilon_n}(x,u_n)\left(x-x_n,\nabla \varphi\right)\\
\leq& C\int_{\mathcal{A}_n^3}\left(|\nabla u_n|^p+|u_n||\nabla u_n|^{p-1}|\nabla\varphi|+|u_n|^p+|u_n|^{p^*}+|u_n|^q\right)\\
\leq& C\sigma_n^{-\frac{N-p}{p}}.
\end{aligned}
\end{equation}
On the other hand, considering Proposition \ref{prop2-1} and the fact that $|\epsilon_n(x-x_n,\nabla dist(\epsilon_n x,\mathcal{M}))|\leq C\sigma_n^{-\frac{1}{p}}$, we have
\begin{equation}\label{4-24}
\begin{aligned}
&F_{\epsilon_n}(x,u_n)-\frac{1}{p^*}u_nf_{\epsilon_n}(x,u_n)+\frac{1}{N}(x-x_n,\nabla_xF_{\epsilon_n}(x,u_n))\\
\geq& C|u_n|^r|m_{\epsilon_n}(x,u_n)|^{q-r}-C.
\end{aligned}
\end{equation}
By Lemma \ref{le4-3}, $x_n\in(\mathcal{M}_{\epsilon_n})^{\sigma_n^{-\frac{1}{p}}}$, thus $\forall x\in B_n$, $x\in (\mathcal{M}_{\epsilon_n})^{C\sigma_n^{-\frac{1}{p}}}$. By Proposition \ref{prop2-1}, we have
$$
|m_{\epsilon_n}(x,u_n)|\geq\epsilon_n^{-1}e^{-dist(\epsilon_n x,\mathcal{M})}=\epsilon_n^{-1}e^{-\epsilon_ndist(x,\mathcal{M}_{\epsilon_n})}\geq1,
$$
therefore
$$
|u_n|^r|m_{\epsilon_n}(x,u_n)|^{q-r}\geq|u_n|^r-1,~~~~\hbox{ for } x\in B_n.
$$
Besides for any $\varepsilon>0$
\begin{equation}\label{4-25}
\begin{aligned}
&\left|\int_{B_n}V(\epsilon_n x)|u_n|^p\varphi+\frac{\epsilon_n}{p}\int_{B_n}(x-x_n,\nabla V(\epsilon_n x))|u_n|^p\varphi\right|\\
\leq& C\int_{B_n}|u_n|^p\varphi\leq \varepsilon\int_{B_n}|u_n|^r+C\sigma_n^{-\frac{N}{p}},
\end{aligned}
\end{equation}
and by Lemma \ref{le4-4}, we have
\begin{equation}\label{4-26}
\begin{aligned}
\int_{B_n}\lambda_n\xi_{\epsilon_n}|u_n|^p\varphi+\frac{\lambda_n}{p}\int_{B_n}|(x-x_n,\nabla\xi_{\epsilon_n})||u_n|^p\varphi\leq C\sigma_n^{-\frac{N-p}{p}}.
\end{aligned}
\end{equation}
Denote $D_n=B_{L\sigma_n^{-1}}(x_n)=\{x||x-x_n|\leq L\sigma_n^{-1}\}$, where $L$ is large enough so that
$$
\int_{B_L(0)}|U_{k_\infty}|^r=m>0,
$$
then for $n$ large enough, $D_n\subset B_n$, and since $\sigma_n^{-\frac{N-p}{p}}u_n(\sigma_n^{-1}x+x_n)\rightharpoonup U_{k_\infty}$, we have
\begin{equation}\label{4-27}
\int_{B_n}|u_n|^r\varphi\geq \frac{m}{4}\sigma_n^{\frac{N-p}{p}r-N}.
\end{equation}
Combining \eqref{4-24}, \eqref{4-25}, \eqref{4-26} and \eqref{4-27}, we estimate the left hand of \eqref{4-22} as follows
\begin{equation}\label{4-28}
\begin{aligned}
&\int_{B_n}\left(NF_{\epsilon_n}(x,u_n)-\frac{N-p}{p}u_nf_{\epsilon_n}(x,u_n)\right)\varphi+\int_{B_n}\left(x-x_n,\nabla_xF_{\epsilon_n}(x,u_n)\right)\varphi\\
&-\int_{B_n}\left(V(\epsilon_n x)+\lambda_n\xi_{\epsilon_n}\right)|u_n|^p\varphi-\frac{\epsilon_n}{p}\int_{B_n}\left(x-x_n,\nabla V(\epsilon_n x)\right)|u_n|^p\varphi\\
&-\frac{\lambda_n}{p}\int_{B_n}\left(x-x_n,\nabla\xi_{\epsilon_n}\right)|u_n|^p\varphi\\
\geq &C\int_{B_n}|u_n|^r\varphi-C\sigma_n^{-\frac{N-p}{p}}\geq C\sigma_n^{\frac{N-p}{p}r-N}.
\end{aligned}
\end{equation}
Considering $r>q^*-1$, we deduce a contraction with \eqref{4-23}.
\end{proof}
\noindent\textbf{Proof of Proposition \ref{prop4-1}. } Let $\epsilon_n\to 0$ as $n\to\infty$. From Proposition \ref{prop4-3}, we have that for any $\varepsilon>0$, there exists $\delta>0$ independent of $n$ such that
\begin{equation}\label{4-29}
\sup\limits_{y\in\mathbb{R}^N}\int_{B_\delta(y)}|u_n|^{p^*}<\varepsilon.
\end{equation}
Since $W=|u_n|$ satisfies
$$
-\Delta_p W\leq CW^{p^*-1}~~\hbox{ in } \mathbb{R}^N,
$$
by \eqref{4-29} and Proposition A.5 of \cite{zll}, we deduce that there exists $M>0$ such that
$$
\sup\limits_{x\in\mathbb{R}^N}|u_n(x)|\leq M,
$$
that is Proposition \ref{prop4-1}.\hfill{$\Box$}
\section{Proof of Theorem \ref{th1-1}}\label{5}
Now the situation is similar to the subcritical case. Assume that the sequence $\{u_n\}$ has the profile decomposition
\begin{equation}\label{5-1}
u_n(x)=\sum\limits_{k\in\Lambda_1}U_k(x-x_{n,k})+r_n.
\end{equation}
Denote
$$
\Omega_R^{n}=\mathbb{R}^N\setminus\left\{\bigcup_{k} B_R(y_{n,k})\cup B_R(0)\right\}.
$$
\begin{lemma}\label{le5-1}
Assume that the profile decomposition \eqref{5-1} holds. Denote $x_k^*=\lim\limits_{n\to\infty}\epsilon_n x_{n,k}$. Then $x_k^*\in\mathcal{A}$, i.e., $x_k^*$ is a critical point of $V$ in $\overline{\mathcal{M}}$.
\end{lemma}
\begin{proof}
By \eqref{5-1}, we have $\|u_n\|_{L^{p^*}(\Omega_R^{(n)})}=O_R(1)$, by Moser's iteration, we have $\|u_n\|_{L^\infty(\Omega_R^{(n)})}=O_R(1)$. Denote $w_n=|u_n|$, then $w_n$ satisfies
\begin{equation}\label{5-2}
-\Delta_p w_n+w_n^{p-1}+\lambda_n\xi_{\epsilon_n}w_n^{p-1}\leq 0.
\end{equation}
Taking $\varphi=w_n\phi^p$ as a test function in \eqref{5-2}, where $\phi(x)=1$ for $x\in\Omega_{R+1}^{(n)}$, $\phi(x)=0$ for $x\notin\Omega_{R}^{(n)}$, $|\nabla \phi|\leq 2$. Then we get that
$$
\int_{\mathbb{R}^N}|\nabla w_n|^p\phi^p+\int_{\mathbb{R}^N}w_n^p\phi^p+\lambda_n\int_{\mathbb{R}^N}\xi_{\epsilon_n}w_n^p\phi^p\leq C\int_{\mathbb{R}^N}|\nabla w_n|^{p-1}\phi^{p-1}w_n|\nabla \phi|,
$$
thus
\begin{equation}\label{5-3}
\int_{\Omega_{R+1}^{(n)}}|\nabla w_n|^p+\int_{\Omega_{R+1}^{(n)}}w_n^p+\lambda_n\int_{\Omega_{R+1}^{(n)}}\xi_{\epsilon_n}w_n^p\leq C\int_{\Omega_R^{(n)}\setminus\Omega_{R+1}^{(n)}}w_n^p,
\end{equation}
which yields that
\begin{align*}
\int_{\Omega_{R+1}^{(n)}}w_n^p\leq C \int_{\Omega_R^{(n)}\setminus\Omega_{R+1}^{(n)}}w_n^p,
\end{align*}
that is
$$
\int_{\Omega_{R+1}^{(n)}}w_n^p\leq\frac{C}{C+1}\int_{\Omega_{R}^{(n)}}w_n^p.
$$
By iteration, we have
\begin{equation}\label{5-4}
\int_{\Omega_R^{(n)}}w_n^p\leq Ce^{-\alpha R},
\end{equation}
where $\alpha>0$ satisfies that $e^{-\alpha}=\frac{C}{C+1}$. Returning to \eqref{5-3}, we have
\begin{equation}\label{5-5}
\int_{\Omega_R^{(n)}}|\nabla w_n|^p+\lambda_n\int_{\Omega_R^{(n)}}\xi_{\epsilon_n}w_n^p\leq Ce^{-\alpha R}.
\end{equation}
Since $dist(x_{n,k},\mathcal{M}_{\epsilon_n})\leq c<\infty$, we have $x_k^*=\lim\limits_{n\to\infty}\epsilon_nx_{n,k}\in\overline{\mathcal{M}}$. If $x_k^*\notin \mathcal{A}$, then $t_k=\nabla V(x_k^*)\neq 0$, we deuce that there exists $\delta_1>0$ such that for $x\in B_{\delta_1}(x_k^*)$,
\begin{equation}\label{5-6}
(t_k,\nabla V(x))\geq\frac{1}{2}|t_k|^2>0,
\end{equation}
and by assumption $(V_2)$, $\delta_1$ may be chosen so small that
\begin{equation}\label{5-7}
(t_k, \nabla dist(x, \mathcal{M}))\geq 0.
\end{equation}
Set
$$
\delta_2=\min\left\{|y_k^*-y_l^*|\left|y_k^*\neq y_l^*,~~k,~~l\in\Lambda_1\right.\right\}.
$$
Choose
$$
0<\delta<\min\left\{\frac{1}{2}\delta_1,\frac{1}{100}\delta_2\right\}.
$$
Denote
$$
\tilde{B}_n=\{x||x-x_{n,k}|\leq 2\delta\epsilon_n^{-1}\},
$$
$$
T_n=\{x|\delta\epsilon_n^{-1}\leq|x-x_{n,k}|\leq2\delta\epsilon_n^{-1}\},
$$
then $T_n\subset\Omega_{\delta\epsilon_n^{-1}}^{(n)}$. By \eqref{5-4}, \eqref{5-5}, we have
\begin{equation}\label{ad5-1}
\int_{T_n}w_n^{p}\leq Ce^{-\alpha\delta\epsilon_n^{-1}},
\end{equation}
\begin{equation}\label{5-8}
\int_{T_n}|\nabla w_n|^p+\lambda_n\int_{T_n}\xi_{\epsilon_n}w_n^p\leq Ce^{-\alpha\delta\epsilon_n^{-1}}.
\end{equation}
Choose $\eta\in C_0^\infty(\mathbb{R}^N)$ such that $\eta(x)=1$ for $|x-x_{n,k}|\leq \delta\epsilon_n^{-1}$; $\eta(x)=0$ for $|x-x_{n,k}|\geq 2\delta\epsilon_n^{-1}$ and $|\nabla \eta|\leq\frac{2}{\delta}\epsilon_n$. As in \eqref{4-14}, we have
\begin{equation}\label{5-9}
\begin{aligned}
&\frac{\epsilon_n}{p}\int_{\tilde{B}_n}(t_k,\nabla V(\epsilon_n x))|u_n|^p\eta+\frac{\lambda_n}{p}\int_{\tilde{B}_n}(t_k,\nabla \xi_{\epsilon_n})|u_n|^p\eta-\int_{\tilde{B}_n}\left(t_k,\nabla_xF_{\epsilon_n}(x,u_n)\right)\eta\\
=&\int_{T_n}|\nabla u_n|^{p-2}(t_k,\nabla u_n)(\nabla u_n,\nabla \eta)-\frac{1}{p}\int_{T_n}|\nabla u_n|^p(t_k,\nabla \eta)\\
&-\frac{1}{p}\int_{T_n}\left(V(\epsilon_n x)+\lambda_n\xi_{\epsilon_n}\right)|u_n|^p(t_k,\nabla \eta)+\int_{T_n}F_{\epsilon_n}(x,u_n)(t_k,\nabla \eta).
\end{aligned}
\end{equation}
By \eqref{ad5-1} and \eqref{5-8}, the right hand side of \eqref{5-9} can be estimated as follows
\begin{equation}\label{5-10}
\begin{aligned}
&\int_{T_n}|\nabla u_n|^{p-2}(t_k,\nabla u_n)(\nabla u_n,\nabla \eta)-\frac{1}{p}\int_{T_n}|\nabla u_n|^p(t_k,\nabla \eta)\\
&-\frac{1}{p}\int_{T_n}\left(V(\epsilon_n x)+\lambda_n\xi_{\epsilon_n}\right)|u_n|^p(t_k,\nabla \eta)+\int_{T_n}F_{\epsilon_n}(x,u_n)(t_k,\nabla \eta)\\
\leq& Ce^{-\tilde{\alpha}\delta\epsilon_n^{-1}}.
\end{aligned}
\end{equation}
On the other hand, from \eqref{5-6} and \eqref{5-7}, we obtain that
\begin{equation}\label{5-11}
\begin{aligned}
&\frac{\epsilon_n}{p}\int_{\tilde{B}_n}(t_k,\nabla V(\epsilon_n x)|u_n|^p\eta+\frac{\lambda_n}{p}\int_{\tilde{B}_n}(t_k,\nabla \xi_{\epsilon_n})|u_n|^p\eta-\int_{\tilde{B}_n}(t_k,\nabla_xF_{\epsilon_n}(x,u_n)\eta\\
\geq& \frac{\epsilon_n}{p}\int_{\tilde{B}_n}(t_k,\nabla V(\epsilon_n x)|u_n|^p\eta+\frac{\lambda_n}{p}\int_{\tilde{B}_n}(t_k,\nabla \xi_{\epsilon_n})|u_n|^p\eta\\
\geq&C\epsilon_n\int_{B_{\delta\epsilon_n^{-1}}(x_{n,k})}|u_n|^p\geq C\epsilon_n.
\end{aligned}
\end{equation}
By \eqref{5-10} and \eqref{5-11}, we deduce a contradiction.
\end{proof}
\noindent\textbf{Proof of Theorem \ref{th1-1}. } Since $\mathcal{A}$ is a compact subset of $\mathcal{M}$,
$$
dist(\mathcal{A},\partial\mathcal{M})>0.
$$
Choosing $0<\delta<dist(\mathcal{A},\partial\mathcal{M})$, by Lemma \ref{le5-1}, we obtain that
\begin{equation}\label{5-12}
\lim\limits_{n\to\infty}\int_{\mathbb{R}^N}\xi_{\epsilon_n}|u_n|^p=0.
\end{equation}
Also by \eqref{5-4}, we have that for $R$ large
$$
\int_{\mathbb{R}^N\setminus\{\cup B_R(x_{n,k})\cup B_R(0)\}}|u_n|^p\leq Ce^{-\alpha R}.
$$
By Moser's iteration,
\begin{equation}\label{5-13}
|u_n(x)|\leq Ce^{-\alpha R},~~~~\hbox{ for } x\in\mathbb{R}^N\setminus\{\cup B_R(x_{n,k})\cup B_R(0)\}.
\end{equation}
Denote $R_n(x)=\min \limits_{k\in\Lambda_1}\{|x|,|x-x_{n,k}|\}$, considering Proposition \ref{prop4-1}, we obtain
$$
|u_n(x)|\leq Ce^{-\alpha R_n(x)}.
$$
Since $\epsilon_nx_{n,k}\to x_k^*\in\mathcal{A}$, for any $\delta>0$, there exists  $\bar{\epsilon}$ such that for $\epsilon_n<\bar{\epsilon}$, $dist(\epsilon_nx_{n,k},\mathcal{A})<\delta$, hence $R_n(x)\geq dist(x,(\mathcal{A}^\delta)_{\epsilon_n})$ and
\begin{equation}\label{5-14}
\begin{aligned}
|u_n(x)|\leq&Ce^{-\alpha dist(x,(\mathcal{A}^\delta)_{\epsilon_n})}\\
\leq &Ce^{-\alpha dist(x,\mathcal{M}_{\epsilon_n})}.
\end{aligned}
\end{equation}
By $(2)$ of Proposition \ref{prop2-1}, $u_n$ is a solution of \eqref{1-2}, thus by Theorem \ref{th3-2}, we deduce that \eqref{1-1} admits infinitely many solutions.\hfill{$\Box$}

\end{document}